\documentclass[10pt]{amsart}

\usepackage[pagewise]{lineno}
\usepackage{amsfonts,amssymb,amsmath,amsthm}
\usepackage{graphicx}
\usepackage{xypic}
\usepackage{indentfirst}
\usepackage{enumerate}
\usepackage{setspace}
\usepackage{tikz-cd}
\usepackage{hyperref}
\usepackage{comment}

\usepackage{datetime}

\usepackage{pgfplots}
\pgfplotsset{width=10cm,compat=1.9}

\usepackage{enumitem} 
\usepackage[left=1.3in, right=1.3in, top=1.2in, bottom=1.2in]{geometry}
\usepackage[toc]{appendix}
\usepackage{mathtools}
\usepackage{mathabx,epsfig}
\usepackage{mathrsfs} %mathscr font

\def\acts{\mathrel{\reflectbox{$\lefttorightarrow$}}}

\usepackage{xy}
\input xy
\xyoption{all}

\newcommand{\Z}{\mathbb{Z}}
\newcommand{\uZ}{\underline{\mathbb{Z}}}
\newcommand{\al}{a_{\lambda}}
\newcommand{\aal}{a_{\alpha}}
\newcommand{\atal}{a_{2\alpha}}
\newcommand{\ats}{a_{2\sigma}}
\newcommand{\ul}{u_{\lambda}}
\newcommand{\uta}{u_{2\alpha}}
\newcommand{\uts}{u_{2\sigma}}
\newcommand{\ual}{u_{\alpha}}
\newcommand{\HPt}{H^{\Phi_2}}
\newcommand{\HPf}{H^{\Phi_4}}
\newcommand{\Tre}{\langle \tr(e_{3\alpha})\rangle[\uta^{-1},\aal^{-1}]}
\newcommand{\tre}{\tr(e_{3\alpha})}
\newcommand{\Si}{\Sigma^{-1}}
\newcommand{\xr}[2]{\underset{#2}{\xrightarrow{#1}}}
\newcommand{\F}{\mathscr{F}}
\newcommand{\PP}{\mathscr{P}}
\newcommand{\eal}{e_{\alpha}}
\newcommand{\etal}{e_{2\alpha}}
\newcommand{\ethal}{e_{3\alpha}}
\newcommand{\res}{res^4_2}
\newcommand{\tr}{tr^4_2}
\newcommand{\bs}{\bigstar}

\theoremstyle{plain}

\newtheorem{Thm}{Theorem}
\newtheorem*{Thm*}{Theorem}
\newtheorem{Lem}[Thm]{Lemma}
\newtheorem{Prop}[Thm]{Proposition}
\newtheorem{Cor}[Thm]{Corollary}

\newtheorem{Rem}{Remark}

	\title{The $RO(C_{2^n})$-graded homotopy of $H\underline{\mathbb{Z}}$ through generalized Tate squares}
	\author{Guoqi Yan}
	\address{}
	\email{gyan@nd.edu}

\date{\today}

\begin{document}

\begin{abstract}
 We propose a new method to compute the $C_{2^n}$-equivariant homotopy groups of the Eilenberg-Mac Lane spectrum $H\underline{\mathbb{Z}}$ as a $RO(C_{2^n})$-graded Green functor using the generalized Tate squares. As an example, we completely compute the $C_4$ case and investigate two $\mathscr{P}$-homotopy limit spectral sequences for the family $\mathscr{P}=\{e,C_2\}$.
\end{abstract}

\maketitle
\tableofcontents

\section{Introduction}
\subsection{\texorpdfstring{$RO(C_{2^n})$}{roc2n}-graded homotopy of \texorpdfstring{$H\uZ$}{HZ}}
Computation of the $RO(G)$-graded homotopy of $H\uZ$ for a finite group $G$ and the constant Mackey functor $\uZ$ has been a meaningful but difficult task. Recently, the work of \cite{MNN19},\cite{Angel22} and \cite{Liu22} reduced this task to $p$-groups. Cyclic groups $C_{p^n}$ of prime power order are among the easiest $p$-groups. The case $C_p$ for odd $p$ is computed by Lewis and Stong, and the $p=2$ case is computed in \cite{Duggerc2} and also \cite{Gre18}. \cite{Zeng17} completely computed $C_{p^2}$ for an odd prime $p$ and suggested methods for $p=2$ and bigger $n$. The additive structure is also considered in \cite{basu21}. Cyclic $2$-groups $C_{2^n}$ have special importance due to the great work of \cite{HHRa}. The regular-representation-suspended $H\uZ$ appears as the slices of various important $G$-spectra in the work of \cite{HHRa},\cite{HHRb},\cite{c4ht4}, etc. For example, one of the key techical result of \cite{HHRa} is

\begin{Thm}(The slice theorem)
The $C_8$-spectrum $P^n_nMU^{((C_8))}$ is contractible if $n$ is odd. It is weakly equivalent to $H\uZ\wedge W$ if $n$ is even, where $W$ is a wedge of isotropic slice cells.
\end{Thm}

As a result, the $RO(C_{2^n})$-graded homotopy of $H\uZ$ serves as the input of the $RO(C_{2^n})$-graded slice spectral sequences of various important $C_{2^n}$-spectra. This motivates us to consider the $G=C_4$ case as a starting point and the same method applies to $C_{2^n}$ for bigger $n$. Actually, our computation also suggests the computations at odd primes as we will explain below.

The $RO(C_4)$-graded homotopy of $H\uZ$ has been considered by several different authors, see \cite{HHRb}, \cite{Zeng17} and \cite{NickG}. \cite{HHRb} computed partial information of the Green functor. \cite{Zeng17} computed $RO(C_{p^2})$-graded homotopy of $H\uZ$ for odd prime $p$ and pointed out a method to compute the case $p=2$, without carrying it out due to the amount of labor needed. He actually considered the $p=2$ case as a more refined case than the odd primary case, while we will try to simplify the computation by taking another route. We observe that our answer, when restricted to orientable representation gradings, is consistent with the odd primary case. \cite{NickG} produced computer programs to carry out the computation of the entire Green functor. His computation is based on considering the box product of cellular chains of different representation spheres, which could be too complicated to be carried out by hand.

\subsection{Our new perspectives}
The complexity of $RO(C_{p^n})$-graded homotopy of $H\uZ$ grows drastically with respect to $n$: As $n$ grows, new irreducible representations are introduced and new levels of the Mackey functors are added. Thus it will become impractical to carry out degree-wise computations, which are mostly used in cellular computations. Our philosophy in this paper is that, instead doing degree-wise computations, we should focus on the global multiplicative structure. The way to realize this idea is through the filtration given by the following chain of families of subgroups
\begin{equation}
    \emptyset\subset \{e\}\subset\{e,C_p\}\subset\cdots\subset \{e,\cdots,C_{p^{n-1}}\}=\PP\subset \mathscr{A}ll.\label{chainoffamily}
\end{equation}
Denote $\mathscr{F}_i=\{e,C_p,\cdots,C_{p^i}\},0\leq i\leq n$, and $\F_{-1}=\emptyset$. 

For any odd prime $p$, the irreducible orthogonal representations of $C_{p^n}=\langle\gamma\rangle$ are: the trivial one-dimensional representation $1$ and the two-dimensional ones $\lambda(m)$ where $\gamma$ acts by rotation of $\frac{2m\pi}{p^n},0<m<2^{n-1}$. The homotopy type of $S^{\lambda(m)}\wedge H\uZ$ only depends on the $p$-adic valuation of $m$ \cite{HHRb}. Thus for odd primes we can consider 
\[
RO(C_{p^n})=\Z\{1,\lambda_{n-1},\lambda_{n-2},\cdots,\lambda_1,\lambda_0\},
\]
where $\lambda_{j}=\lambda(p^j)$ is the representation with isotropy $C_{p^j}$. Thus we can write $\lambda_n=2$ if we prefer. For $p=2$, we use the notation $\lambda_{n-1}=2\alpha$, where $\alpha$ is the one-dimensional real sign representation. In this case we have
\[
RO(C_{2^n})=\Z\{1,\alpha,\lambda_{n-2},\cdots,\lambda_1,\lambda_0\},
\]

For now, we consider all primes at the same time and use the convention $\lambda_{n-1}=2\alpha$ when $p=2$. Because of the models $\widetilde{E\mathscr{F}_i}\simeq S^{\infty\lambda_i}$ (note $\widetilde{E\emptyset}\simeq S^0$ and $\widetilde{E\mathscr{A}ll}\simeq *$), the above chain (\ref{chainoffamily}) realizes as
\begin{equation}
H\to H[a_{\lambda_0}^{-1}]\to H[a_{\lambda_1}^{-1}]\to \cdots\to H[a_{\lambda_{n-1}}^{-1}]\to *,\label{toweroflocalizations}
\end{equation}
where we use $H$ for $H\uZ$ for a shorter notation, $H[a_{\lambda_i}^{-1}]=H\wedge \widetilde{E\mathscr{F}_i}$, and $a_{\lambda_{i}}:S^0\to S^{\lambda_i}$ is the Hurewicz image of the map that is the inclustion of fixed-points on the space level, which is usually called the Euler class. This provides us with a tower of localizations, and we can 'delocalize' the tower to get $\pi_{\bs}H$.

To realize our philosophy further, we present our result in a concise and clean way so that some global phenomenon can be easily seen. Actually, the clean presentation helps us discover the localization theorem proved in another paper \cite{LocThm}. It also enables us to read off $H[a_{\lambda_i}^{-1}]_{\bs}$ directly from the knowledge of $H_{\bs}$. The $C_4$ case is shown in section \ref{ringstructure}. It is especially useful for people who mainly care about the effect of multiplication by Euler classes.

Another novelty of our paper is that we consider the generalized Tate square for families beyond the trivial family $\{e\}$. In the $C_4$ case, we considered the family $\PP=\{e,C_2\}$. Along this line, we can use general families $\F_i$ in (\ref{chainoffamily}) to do computations for bigger $n$. The spectral sequence for the trivial family $\{e\}$ (known as the homotopy fixed-point spectral sequence) used in \cite{Zeng17} has no differentials. The $\PP$-homotopy limit spectral sequence does not collapse and the differentials in it may suggest general patterns about the $\F_i$-homotopy limit spectral sequences for bigger $n$.

A big advantage of using the $\F_i$-Tate square, and the $\F_i$-homotopy limit spectral sequence, though not obvious in this paper, is that it enables us to do computations of homotopy of Eilenberg-Mac Lane spectra $H\underline{M}$ for general $C_{p^n}$-Mackey functors, not only the constant ones. The key fact used in the $H\uZ$ computation in \cite{Zeng17} is that since $\uZ$ is constant,
\[
(H\uZ_{C_{p^n}})^{C_{p^i}}\cong H\uZ_{C_{p^n}/C_{p^i}}.
\]
Here we use $H\uZ_{K}$ to denote the $K$-equivariant $H\uZ$. This fact enables him to deduce $\underline{\pi}_VH\uZ$ from quotient groups when $V$ does not contain copies of $\lambda_0$. This deduction no longer works for non-constant Mackey functors. For example, $(H\underline{\mathbb{A}}_{C_4})^{C_2}$ is not equivalent to $H\underline{\mathbb{A}}_{C_4/C_2}$, where $\underline{\mathbb{A}}$ is the Burnside Green functor. With the help of the generalized Tate squares, however, the author is able to use the comparison map $H\underline{\mathbb{A}}_{C_4}\to H\underline{\mathbb{Z}}_{C_4}$ to deduce the $RO(C_4)$-graded homotopy of the former spectrum.

The $RO(C_{p^n})$-graded homotopy Green functor can be understood as different layers of $RO(C_{p^n})$-graded commutative rings indexed by the subgroups, connected by restrictions and transfers. Even though we are computing layer-by-layer using Tate squares in this paper, the Mackey functor structure is not lost. We can consider $RO(C_{p^n})\text{'}\subset\text{'} RO(C_{2^n})$ as the index $2$ subgroup of orientable representations, thus focus on the $p=2$ case. In this case, $res^{C_{2^i}}_{C_{2^{i-1}}}$ and $tr^{C_{2^i}}_{C_{2^{i-1}}}$ are already know for $i\leq  n-1$. $res^{C_{2^n}}_{C_{2^{n-1}}}$ and $tr^{C_{2^n}}_{C_{2^{n-1}}}$ can be deduced from the following two cofiber sequences
\begin{equation}
	\begin{split}
	&S^{-1}\xrightarrow{\aal}S^{\alpha-1}\to C_{2^n}/{C_{2^{n-1}}}_+\xrightarrow{res}S^0\xrightarrow{\aal}S^{\alpha},\\
	&S^1\xleftarrow{\aal}S^{1-\alpha}\xleftarrow{}C_{2^n}/{C_{2^{n-1}}}_+\xleftarrow{tr}S^0\xleftarrow{\aal}S^{-\alpha}.\label{restr}	
	\end{split}
\end{equation}
Here the second cofiber sequence is obtained from the first one by taking Spanier-Whitehead dual, and the maps $res,tr$ are the maps that will become $res^{C_{2^n}}_{C_{2^{n-1}}}$ and $tr^{C_{2^n}}_{C_{2^{n-1}}}$ after applying the contravariant functors $[-,S^{-V}\wedge H]^{C_{2^n}}, V\in RO(C_{2^n})$.

Now we introduce our method in more details. The main tool we will use is the generalized Tate square for a family of subgroups $\F$ and an arbitrary $G$-spectrum $E$:
\begin{equation}
	\xymatrix{
		E\F_+\wedge E\ar[d]_{\simeq}\ar[r] &E\ar[d]\ar[r] &\widetilde{E\F}\wedge E\ar[d]\\
		E\F_+\wedge F(E\F_+,E)\ar[r] &F(E\F_+,E)\ar[r] &\widetilde{E\F}\wedge F(E\F_+,E)
	}.\label{FTate}
\end{equation}
Here $G$ can be any compact Lie group, but in this paper we are only concerned with $G=C_{2^n}$. $E\F$ is the universal $G$-space characterized up to $G$-equivalence by \cite[p.44-45]{Alaska}
\[E\F^H=
\begin{cases}
	*,\quad H\in \F,\\
	\emptyset,\quad \text{otherwise}.
\end{cases}
\]
$\widetilde{E\F}$ is defined by the following cofiber sequence:
\[
E\F_+\to S^0\to \widetilde{E\F}.
\]
That the vertical map on the left is an equivalence of $G$-spectra is proved in \cite[Prop 17.2, p.100]{GM95}. We will use this canonical isomorphism to identify elements in $\pi_{\bigstar}(E\F_+\wedge E)$ with that of $\pi_{\bigstar}(E\F_+\wedge F(E\F_+,E))$, which we will compute from the second row.

Every two successive terms in the tower (\ref{toweroflocalizations}) fit in a generalized Tate square
\begin{equation}
	\xymatrix{
		E{\F_i}_+\wedge \widetilde{E{\F_{i-1}}}\wedge H\ar[d]_{\simeq}\ar[r] &\widetilde{E{\F_{i-1}}}\wedge H\ar[d]\ar[r] &\widetilde{E{\F_i}}\wedge H\ar[d]\\
		E{\F_i}_+\wedge \widetilde{E{\F_{i-1}}}\wedge H\ar[r] &F(E{\F_i}_+,\widetilde{E{\F_{i-1}}}\wedge H)\ar[r] &\widetilde{E{\F_i}}\wedge F(E{\F_i}_+,\widetilde{E{\F_{i-1}}}\wedge H)
	}.
\end{equation}
if we take $\F=\F_{i}$ and $E=\widetilde{E{\F_{i-1}}}\wedge H$. Here we used the fact that $\widetilde{E{\F_{i-1}}}\wedge \widetilde{E{\F_{i}}}\simeq \widetilde{E{\F_{i}}}$, which can be easily checked by looking at fixed-points on the space level. The celluar filtration of $E{\F_i}_+\simeq S(\infty\lambda_{i})_+$ gives us a spectral sequence computing the $RO(G)$-graded homotopy of $F(E{\F_i}_+,\widetilde{E{\F_{i-1}}}\wedge H)$. This spectral sequence has three different names as in \cite[Prop 2.24]{MNN19}. It is also a $a_{\lambda_i}$-Bockstein spectral sequence. We choose the name the $\F_i$-homotopy limit spectral sequence from \cite{MNN19}. Note that the cells of $E{\F_i}_+$ all have orbit type $C_{2^n}/C_{2^i}$, thus it only requires knowledge of the $C_{2^i}$-equivariant homotopy $\pi_{\bs}^{C_{2^i}/C_{2^i}}\widetilde{E{\F_{i-1}}}\wedge H$, which is known by induction. Since it is always a $\Z/2$-algebra as a localization at the Euler class of the sign representation of $C_{2^i}$, the Weyl actions are trivial. The homotopy of $\widetilde{E{\F_i}}\wedge H$ is computed by induction and the initial case $\widetilde{E{\PP}}\wedge H$ is computed as \cite[Prop 3.18, p.42]{HHRa}
\begin{equation}
\pi_{\bs}^{G/G}\widetilde{E{\PP}}\wedge H=\Z/2[\uta,\aal^{\pm},a_{\lambda_0}^{\pm},\cdots,a_{\lambda_{n-2}}^{\pm}].\label{geometric}
\end{equation}

After we have a complete understanding of $\pi_{\bigstar}^{G/G}H$, we work backwards and investigate the $\PP$-homotopy limit spectral sequence for $\pi_{\bigstar}^{G/G}F(E\PP_+,H)$. Computation of this spectral sequence motivates us to prove a localization theorem in \cite{LocThm}. It applies to much more general situations, but in our setting where $G=C_{p^n}$, $p$ any prime, it states
\[
F(E{\F_i}_+,H)\simeq H[u_{\lambda_i}^{-1}].
\]

\addtocontents{toc}{\protect\setcounter{tocdepth}{0}}
 
\subsection*{Organization} The structure of the paper is as follows: In section $2$ we do some computations in the main Tate square. In section $3$ we compute the $\PP$-homotopy limit spectral sequence for $\HPt$ and use the $\PP$-Tate square to determine its homotopy. In section $4$ we complete the computation of the additive structure of $\pi^{C_4/C_4}_{\bigstar}H\uZ$. In section $5$ we determine the multiplications of this $RO(C_4)$-graded commutative ring. In section $6$ we give a complete list of the Mackey functor structure. In section 7, we compute the $\PP$-homotopy limit spectral sequence for $\pi_{\bigstar}^{C_4/C_4}F(E\PP_+,H)$. In Appendix A, we compare our computation with the computer-based computation by Nick Georgakopoulos. In Appendix B, we discuss the $RO(G)$-graded commutativity issue and show that any two classes in $\pi^{C_4/C_4}_{\bigstar}H\uZ$ commute on the nose without a sign. 

\subsection*{Acknowledgement.} I want to thank my advisor Mark Behrens for his careful proofreading of the paper and his generous support for my research. I want to thank Mingcong Zeng for a meaningful conversation a year ago and his previous papers on this subject. I also want to thank Jack Carlisle for pointing out some typos and Igor Sikora for many suggestions that improved the readability of the paper.

\addtocontents{toc}{\protect\setcounter{tocdepth}{2}}

\begin{center}
NOTATION
\end{center}
\begin{itemize}
 \item In what follows, $G=C_4$. Its generator is denoted by $\gamma$, i.e., $C_4=\langle \gamma\rangle$. We use $e$ for both the identity element and the trivial subgroup. $tr_2^4$ is short for  $tr_{C_2}^{C^4}$. Similar for $tr_1^2$ and restrictions. As in \cite{NickG}, we use notations like $\overline{\ul}=\res(\ul),\overline{\overline{\ul}}=res^4_1(\ul)$ for simplicity. Double bars mean that we restrict to two levels lower.
 \item $H=H\uZ$.
 \item To avoid confusion, we use $C_2=\langle \gamma^2\rangle$ to denote the subgroup of $C_4$ and $C_2'=C_4/{C_2}=\langle \bar{\gamma}\rangle$ to denote the quotient group. So that we have the short exact sequence $0\to C_2\to C_4\to C_2'\to 0$. The sign representation of $C_2$ is denoted by $\sigma$, and the sign representation of $C_2'$ and $C_4$ is denoted by $\alpha$.
 \item In this paper we will take $RO(C_4)=\Z\{1,\alpha,\lambda\}$. There is another irreducible complex representation $\lambda'$ where $\gamma$ acts by rotation of $\pi/2$ clockwise. But we have $\lambda\cong\lambda'$ as real orthogonal representations, with the isomorphism given by complex conjugation.
 \item We use the plain star $*$ to denote integer grading, $\bigstar=*+*\alpha+*\gamma\in RO(G)$ to denote $RO(G)$-grading. For example, $\underline{\pi}_{*+*\alpha}H\uZ=\underset{m,n\in \Z}{\oplus}\underline{\pi}_{m+n\alpha}H\uZ$, or $\underline{\pi}_{\bigstar}H\uZ=\underset{V\in RO(G)}{\oplus}\underline{\pi}_VH\uZ$.
 \item $[\quad]$ means a polynomial generator while $\langle a,b\rangle$ means additive generators. For example, $\Z/2[c]\langle a,b\rangle$ means we have elements of the form $c^ia,c^ib, i\geq 0$ all of whom are 2-torsion. $\Z/2\langle a^i\rangle_{i\geq 1}\langle b^j\rangle_{j\geq 1}$ means we have elements of the form $a^ib^j,i,j\geq 1$ all of whom are 2-torsion.
 \item $x^{-1}$ means divisibility, i.e., $x^{-1}y$ (or equally $\frac{y}{x}$) means an element that satisfies $x\cdot{x^{-1}y}=y$. $\langle z\rangle[\frac{1}{x}]$ means we have elements $z,\frac{z}{x},\frac{z}{x^2},\frac{z}{x^3},\cdots$.
 \item Note some elements of the form $2a,4b$. In some cases they are just formal symbols, not 2 times of an actual element $a$, or 4 times of an actual element $b$. Their meaning will be explained in context. While in some other cases, for example, in the gold relation $\atal\ul=2\al\uta$, the 2 on the right-hand side indeed mean 2 times of $\al\uta$. Whether it is the first or the second case is easy to tell from context. Sometimes to emphasize the actual multiplication, we will use $c\cdot d$ to mean the multiplication of the actual classes $c$ and $d$. 
 \item All Mackey functors has a underlying bar below it. For a homotopy Mackey functor $\underline{\pi_V}$, we use both ${\pi_V^{G/H}}$ and $\underline{\pi_V}(G/H)$ to denote its value at $G/H$-level. When we write $\pi_V$, we mean $\underline{\pi_V}(G/G)\in Ab$. For example, $\HPt_{\bigstar}=\underline{\pi_{\bigstar}}\HPt(G/G)$. 
 \item We use $i^*_H$ to denote the various pullback functors $i^*_H:Top^G\to Top^H, RO(G)\to RO(H),Sp^GU\to Sp^Hi^*_HU$, where $U$ is a complete $G$-universe. The meaning of this symbol is clear from the context.
 \item Classes we will use are the Euler classes $a_V\in\pi_{-V}^{G/G}H\uZ$ for $V$ an actual representation with $V^G=0$, the orientation classes $u_V\in\pi^{G/G}_{|V|-V}H\uZ$ where $V$ is an actual orientable representation with $V^G=0$, like $2\alpha$ and $\lambda$. We also have the classes $u_{\alpha}\in \pi^{G/C_2}_{1-\alpha}H\uZ$ and its inverse $\eal\in\pi^{G/C_2}_{\alpha-1}H\uZ$. In \cite[Def 3.4 and Lem 3.5]{HHRb}, $e_{\alpha}\cdot u_{\alpha}=\pm 1\in \pi^{G/C_2}_{0}H=\Z$. We choose the $e_{\alpha}$ that takes the positive sign such that we have $e_{\alpha}=u_{\alpha}^{-1}\in \pi^{G/C_2}_{\alpha-1}H$. Similarly, we take $e_{\lambda}$ such that $e_{\lambda}=(res^4_1({\ul}))^{-1}\in \pi^{G/e}_{\lambda-2}H$. Properties of these classes are listed in \cite[Def 3.4 and Lem 3.6]{HHRb}.
 \item We denote $\HPt=H\wedge \widetilde{EC_4}$ and $\HPf=H\wedge \widetilde{E\{e,C_2\}}$. The notations are motivated by
\[
(\HPt)^{C_2}=H^{\Phi C_2},\,(\HPf)^{C_4}=H^{\Phi C_4}.
\]
See \cite[Thm 9.8, p.112]{LMS}.
\end{itemize}

\section{Some computations in the main Tate square}
We start from the Tate square which will be referred as the main Tate square
\begin{equation}
	\xymatrix{
		H_h\ar[d]_{\simeq}\ar[r] &H\ar[d]\ar[r] &\HPt\ar[d]\\
		H_h\ar[r] &H^h\ar[r] &H^t
	}.\label{Tate}
\end{equation}
Here $H_h=EG_{+}\wedge H,H^h=F(EG_+,H),H^t=H^h\wedge\widetilde{EG}$. The computation starts from the homotopy group of the Borel completion $H^h=F(EG_+,H\uZ)$. $EC_4$ admits a good model $S(\infty\lambda)$, whose cellular filtration looks like:
\[\xymatrixcolsep{1pc}
\xymatrix{
	{*}\ar[rr] &&S(\lambda)_+\ar[ld]_{\simeq}\ar[rr] &&S(2\lambda)_+\ar[ld]\ar[rr] &&S(3\lambda)_+\ar[ld]\ar[rr]&&\cdots\\
	&G/e_+\ar@{-->}[ul] &&G/e_+\wedge S^1\ar@{-->}[ul] &&G/e_+\wedge S^2\ar@{-->}[ul]
}.
\]
Apply $\pi_V F(-,H\uZ)$ to the above filtration gives us the homotopy fixed-point spectral sequence.
\begin{Thm} The $RO(C_4)$-graded homotopy fixed point spectral sequence (HFPSS) takes the form \cite[Def 1.4]{Gre18}
\begin{equation}
	E_2^{V,s}=H^s(C_4;\pi_V^{G/e}H)\Rightarrow \pi_{V-s}^{G/G}H^h,\, |d_r|=(r-1,r)\label{HFPSS}
\end{equation}
\end{Thm}

This spectral sequence collapses at $E_2$ since $\pi_V^{G/e}H$ is concentrated in virtual representations with underlying degree $0$, i.e., it is non-zero only when $V=m(1-\alpha)+n(2-\lambda),m,n\in \Z$. Thus the strong convergence is guaranteed by \cite[Thm 7.1]{Boa99} and the remark below it.

When it happens that $V=m(1-\alpha)+n(2-\lambda),m,n\in \Z$, the $m$ even case amounts to compute the group cohomology ring 
\[
H^*(C_4;\Z)=\Z[x]/4x,|x|=2.
\]
The $m$ odd case amounts to compute the group cohomology $H^*(C_4;\widetilde{\Z})$ ($\gamma$ acts as $-1$ on $\widetilde{\Z}$) as a module over $H^*(C_4;\Z)$, it is 
\[
H^*(C_4;\widetilde{\Z})=\Z/2\langle y\rangle[x],|y|=1.
\]

This is because firstly it is $\Z/2$ in odd degrees from the periodic free resolution $F\xrightarrow{\epsilon}\Z$:
\[
\cdots\to\Z[C_4]\xrightarrow{1-\gamma}\Z[C_4]\xrightarrow{N=1+\gamma+\gamma^2+\gamma^3}\Z[C_4]\xrightarrow{1-\gamma}\Z[C_4]\xrightarrow{\epsilon}\Z\to 0.
\]

Secondly, the pairings
\[
H^{2i}(C_4;\Z)\otimes H^1(C_4;\widetilde{\Z})\xrightarrow{\cong} H^{2i+1}(C_4;\Z\otimes\widetilde{\Z})\cong H^{2i+1}(C_4;\widetilde{\Z}),i\geq 0
\]
coming from the diagonal approximation $\Delta:F\to F\otimes F$ are isomorphisms. The components of $\Delta$ are $\Delta_{pq}:F_{p+q}\to F_p\otimes F_q$. In this case it is \cite[Ch.5, p.108]{Brown}
\begin{equation}
	\Delta_{pq}(1)=
	\begin{cases}
		1\otimes 1,\quad p\,\, \text{even}\\
		1\otimes \gamma,\quad p\,\, \text{odd},q\,\,\text{even}\\
		\underset{0\leq i< j\leq 3}{\Sigma}\gamma^i\otimes\gamma^j,\quad p\,\, \text{odd},q\,\,\text{odd}.
	\end{cases}\label{diagapprox}
\end{equation}
Now we return to the HFPSS (\ref{HFPSS}). Firstly, the target $\pi_{\bigstar}H^h$ is $\uta,\ul$-local. Take $\uta$ as an example. We have the shearing isomorphism $G/e_+\wedge S^{2\alpha-2}\cong G/e_+$. The filtration quotients of the above cellular structure of $E{C_4}_+$ are $G/e_+\wedge S^k$. Multiplication by $\uta$ on $H^h$ is given by $S^{2-2\alpha}\wedge H^h\to H\wedge H^h\to H^h$, where the second map is the $H$-module structure of $H^h$. We have
\begin{equation*}
    S^{2-2\alpha}\wedge F(G/e_+\wedge S^k,H)\simeq F(G/e_+\wedge S^{2\alpha-2}\wedge S^k,H)\simeq F(G/e_+\wedge S^k,H).
\end{equation*}
From this we deduce $S^{2-2\alpha}\wedge F( S(k\lambda)_+,H)\cong F(S(k\lambda)_+,H)$ for all $k$ by induction. Passing to homotopy limits proves the claim.

Then
\[
E_2^{0,*}=H^*(C_4;\pi_0^{G/e}H)=\Z[x]/4x
\] 
converges to 
\[
\pi_*H^h=\Z[\frac{\al}{\ul}]/4\al,
\]
and
\[
E_2^{1-\alpha,*}=H^*(C_4;\pi_{1-\alpha}^{G/e}H)=\Z/2\langle y\rangle[x]
\] 
converges to 
\[
\pi_{*-\alpha}H^h=\Z/2\langle \aal\rangle[\frac{\al}{\ul}].
\]

By $\uta,\ul$-periodicity on both sides, we get
\begin{Prop}
The $C_4/C_4$-level of the $RO(C_4)$-graded coefficients of $H^h$ are
\begin{equation}
	H^h_{\bigstar}=\Z[\aal,\al,\uta^{\pm},\ul^{\pm}]/(\atal\ul-2\al\uta,2\aal,4\al).
\end{equation}
\end{Prop}
It additively splits as
\begin{align*}
	H^h_{\bigstar}&=\Z[\al,\uta^{\pm},\ul^{\pm}]/4\al\\
	&\oplus\Z/2\langle \aal\rangle[\al,\uta^{\pm},\ul^{\pm}]
\end{align*}
since $\aal^2=\atal=2\frac{\al\uta}{\ul}$.

Then 
\begin{equation}
	H^t_{\bigstar}=H^h_{\bigstar}[\al^{-1}]=\Z/4[\al^{\pm},\aal,\uta^{\pm},\ul^{\pm}]/(\atal\ul=2\al\uta,2\aal)
\end{equation}
and taking kernels and cokernels of the map
\[
H^h_{\bigstar}\to H^t_{\bigstar}
\]
leaves us with the extension problem
\begin{equation*}
0 \to\left(\begin{array}{c} \Z/4\langle \Si\frac{1}{\al^i}\rangle_{i\geq 1}[\uta^{\pm},\ul^{\pm}]\\
\oplus\Z/2\langle \Si\frac{\aal}{\al^i}\rangle_{i\geq 1}[\uta^{\pm},\ul^{\pm}]\end{array}\right)\to {H_h}_{\bigstar}\to 4\Z[\uta^{\pm},\ul^{\pm}]\to 0.
\end{equation*}
It splits as the targets are $\Z$'s. Here $4\Z[\uta^{\pm},\ul^{\pm}]$ means its image in $H^h_{\bigstar}$ is $4$ times of the elements with the same names. We conclude:
\begin{Prop}
The $C_4/C_4$-level of the $RO(C_4)$-graded coefficients of $H_h$ are
\begin{equation}
	\begin{aligned}
		{H_h}_{\bigstar}&=4\Z[\uta^{\pm},\ul^{\pm}]\\
		&\oplus\Z/4\langle \Si\frac{1}{\al^i}\rangle_{i\geq 1}[\uta^{\pm},\ul^{\pm}]\\
		&\oplus\Z/2\langle \Si\frac{\aal}{\al^i}\rangle_{i\geq 1}[\uta^{\pm},\ul^{\pm}].
	\end{aligned}\label{H_h}
\end{equation}
\end{Prop}

\section{The \texorpdfstring{$\PP$}{pp}-Tate square and the \texorpdfstring{$\PP$}{pp}-homotopy limit spectral sequence}\label{methodone}
\subsubsection{The \texorpdfstring{$\PP$-homotopy limit spectral sequence}{PP}}

Take $\PP=\{e,C_2\}$, the family of proper subgroups. We have following generalized Tate square:
\begin{equation}
	\xymatrix{
		E\PP_+\wedge \HPt\ar[d]_{\simeq}\ar[r] &\HPt\ar[d]\ar[r] &\HPf\ar[d]\\
		E\PP_+\wedge F(E\PP_+,\HPt)\ar[r] &F(E\PP_+,\HPt)\ar[r] &F(E\PP_+,\HPt)[\aal^{-1}]
	}.\label{FTate2}
\end{equation}

We start the computation from the homotopy of $F(E\PP_+,\HPt)$. Similiar to the homotopy fixed point spectral sequence, we can filter $E\PP_+\simeq S(\infty\alpha)_+$ to get a spectral sequence which converges to the homotopy of $F(E\PP_+,\HPt)$. We will call it the $\PP$-homotopy limit spectral sequence. It is constructed as follows:

We have $E\PP_+\simeq S(\infty\alpha)_+\simeq \varepsilon^{*}EC_2'$, where $\varepsilon:G\to C_2'$ is the natural projection and ${EC_2}'$ is regarded as a $G$-space through this map. Sometimes we suppress the notation and regard $EC_2'$ as a $G$-space without $\varepsilon^*$. The cellular filtration is of the form
\[\xymatrixcolsep{1pc}
\xymatrix{
	{*}\ar[rr] &&S(\alpha)_+\ar[ld]_{\simeq}\ar[rr] &&S(2\alpha)_+\ar[ld]\ar[rr] &&S(3\alpha)_+\ar[ld]\ar[rr]&&\cdots\\
	&G/{C_2}_+\ar@{-->}[ul] &&G/{C_2}_+\wedge S^1\ar@{-->}[ul] &&G/{C_2}_+\wedge S^2\ar@{-->}[ul]
}.
\]
Applying $F(-,\HPt)$ we get
{\tiny
\begin{equation}
	\xymatrixcolsep{0.3pc}
	\xymatrix@C-1em{
		{*}\ar@{-->}[rd]&&F(S(\alpha)_+,\HPt)\ar[ll]\ar@{-->}[dr] &&F(S(2\alpha)_+,\HPt)\ar[ll]\ar@{-->}[rd]&&F(S(3\alpha)_+,\HPt)\ar[ll]\\
		&F(G/{C_2}_+,\HPt)\ar[ur]^{\simeq} &&F(G/{C_2}_+\wedge S^1,\HPt)\ar[ur]&&F(G/{C_2}_+\wedge S^2,\HPt)\ar[ur]
	},\label{tower}
\end{equation}
}
from which we get a spectral sequence with
\[
E_1^{V,s}=\pi_{V-s}(F(G/{C_2}_+\wedge S^s,\HPt))\Rightarrow\pi_{V-s}F(E\PP_+,\HPt),V\in RO(G).
\] 
The strong convergence is again guaranteed by \cite[Thm 7.1]{Boa99} and the remark below it, after we prove this spectral sequence collapses at $E_4$. $d_1$ is given by the following diagram:
{\scriptsize
\[
\xymatrix{
	[G/{C_2}_+,S^{-V}\wedge\HPt]^G\ar[r]^-{d_1}\ar[d]^{\cong}&[\Sigma^{-1}G/{C_2}_+\wedge S^1,S^{-V}\wedge\HPt]^G\ar[r]^-{d_1}\ar[d]^{\cong}&[\Sigma^{-2}G/{C_2}_+\wedge S^2,S^{-V}\wedge\HPt]^G\ar[r]^-{d_1}\ar[d]^{\cong}&\cdots\\
	\underline{\pi_0}(S^{-V}\wedge \HPt)(G/{C_2})\ar[r]^-{1-\gamma}&\underline{\pi_0}(S^{-V}\wedge \HPt)(G/{C_2})\ar[r]^-{1+\gamma}&\underline{\pi_0}(S^{-V}\wedge \HPt)(G/{C_2})\ar[r]^-{1-\gamma}&\cdots\\
}
\]
}
By \cite[3.3, p.38]{HHRa}, this is the cochain $C^*_G(EC_2';\underline{\pi_V}(\HPt))$, thus we get 
\begin{Thm} The $\PP$-homotopy limit spectral sequence takes the form
\begin{equation}
	E_2^{V,s}=H^{s}_G(EC_2';\underline{\pi_V}\HPt)\cong H^s(C_2';{\pi_V^{G/C_2}}\HPt)\Rightarrow \pi_{V-s}(F(E\PP_+,\HPt)).\label{PholimSS}
\end{equation}
\end{Thm}
The second identification of $E_2$ with the cohomology of $C_2'$ is proved in \cite[Prop 23.2, p.129]{GM95}. For the coefficient ${\pi_V^{G/C_2}}\HPt$ as a $C_2'$-module, we use the isomorphism
\[
\pi_V^{G/C_2}\HPt\cong\pi^{C_2/C_2}_{i^*_{C_2}V}(H[\ats^{-1}]).
\]
The right hand side is written in (\ref{C2}).
From this we know the left-hand side is
\[
\pi_{\bigstar}^{G/C_2}\HPt=\Z/2[\overline{\al}^{\pm},\overline{\ul}][\eal^{\pm}].
\]
Since the entire ring is over $\Z/2$, the $C_2'$-action is trivial.
\begin{Lem}
The $E_2$-page of the $\PP$-homotopy limit spectral sequence (\ref{PholimSS}) is
\begin{equation}
	\begin{split}
		&E_2^{\bigstar,*}=\Z/2[\overline{\al}^{\pm},\overline{\ul},\eal^{\pm},b]\quad\text{with}\quad |d_r|=(r-1,r),\\
		&|b|=(0,1),|\overline{\al}|=(-\lambda,0),|\overline{\ul}|=(2-\lambda,0), |e_{\alpha}|=(\alpha-1,0).
	\end{split}
\end{equation}
\end{Lem}
Since each term in (\ref{tower}) is a $\HPt$-module through the maps 
\[
F(S(k\alpha)_+,\HPt)\xleftarrow{}F(S(\infty\alpha)_+,\HPt)\xleftarrow{}F(S^0,\HPt)\simeq \HPt,
\]
the spectral sequence is a spectral sequence of $\HPt_{\bigstar}$-modules. Thus $d_r$ is $\al^{\pm},\ul,\uta$-linear. Here these elements in $\HPt_{\bigstar}$ act on $E_2$ through the restriction map $\res:\pi_V^{G/G}\HPt\to \pi_V^{G/C_2}\HPt$. Since everything is also a $F(S(\infty\alpha)_+,\HPt)$-module, $d_r$ is also $\uta^{-1}=e_{\alpha}^2$-linear. Because as in the last section, using the shearing isomorphism $G/{C_2}_+\wedge S^{2-2\alpha}\cong G/{C_2}_+$, we can show that $F(E\F_+,\HPt)$ is $\uta$-local. Thus in the following, we will use $e_{2\alpha}^{\pm}=\uta^{\mp}$ interchangeably. This spectral sequence also has a multiplicative structure as explained in \cite[p.5]{Gre18}, because of the paring of representation spheres $S^V\wedge S^W\to S^{V+W}$.

\subsubsection{The homotopy of \texorpdfstring{$\HPt$}{HPt}}

Next we consider the differentials. As computed in the following proposition, $b^2\in E_{\infty}$, which implies that $d_r$ is $b^2$-linear too. So the only cases we need to consider are 
\[
d_r(1),d_r(b),d_r(\eal),d_r(\eal b).
\]

\begin{Prop}
The only non-trivial differentials among the four generators listed above are
\[
d_3(b)=\frac{\overline{\ul}}{\overline{\al}}b^4\quad\text{and}\quad d_3(\eal)=\frac{\overline{\ul}}{\overline{\al}}b^3\eal.
\]
\end{Prop}

\begin{proof}
Since we are dealing with a spectral sequence of algebras, $1$ is always a permanent cycle. For degree reason, it cannot be hit by any differential, thus surives to $E_{\infty}^{0,0}$.

For $b$, in the following diagram, $b\in [\Sigma^{-1}G/{C_2}_+\wedge S^1,\HPt]^G$ maps to $[S(2\alpha)_+\wedge S^{-1},\HPt]^G$ and cannot be lifted to $[S(5\alpha)_+\wedge S^{-1},\HPt]^G$ because of (\ref{drb}).
{\scriptsize
\begin{equation}
	\xymatrixcolsep{1pc}
	\xymatrix{
		[S^{-2},\HPt]^G=0 &[S^{-2},\HPt]^G=0\ar[l]&[S^{-2},\HPt]^G=0\ar[l]&[S^{-2},\HPt]^G=0\ar[l]&\cdots\ar[l]\\
		[S^{2\alpha-2},\HPt]^G\ar[u]&[S^{3\alpha-2},\HPt]^G\ar[u]\ar[l]^{\aal}_{\cong}&[S^{4\alpha-2},\HPt]^G\ar[u]\ar[l]^{\aal}&[S^{5\alpha-2},\HPt]^G\ar[u]\ar[l]^{\aal}&\cdots\ar[l]\\
		[S(2\alpha)_+\wedge S^{-1},\HPt]^G\ar[u]^{\cong}&[S(3\alpha)_+\wedge S^{-1},\HPt]^G\ar[u]^{\cong}\ar[l]_{\cong}&[S(4\alpha)_+\wedge S^{-1},\HPt]^G\ar[u]^{\cong}\ar[l]&[S(5\alpha)_+\wedge S^{-1},\HPt]^G\ar[l]\ar[u]^{\cong}&\cdots\ar[l]\\
		[S^{-1},\HPt]^G=0\ar[u] &[S^{-1},\HPt]^G=0\ar[l]\ar[u]&[S^{-1},\HPt]^G=0\ar[l]\ar[u]&[S^{-1},\HPt]^G=0\ar[u]\ar[l]&\cdots\ar[l]
	}.
\end{equation}
}
Here the vertical exact sequences are derived from the cofiber sequences 
\[
S(k\alpha)_+\to S^0\to S^{k\alpha}.
\] 
The horizontal sequences (not exact) are derived from the following commutative diagram 
\[
\xymatrix{
	S(k\alpha)_+\ar[r]\ar[d] &S^0\ar[r]\ar[d] &S^{k\alpha}\ar[d]\\
	S((k+1)\alpha)_+\ar[r] &S^0\ar[r] &S^{(k+1)\alpha}
}.
\]
That the sequence in the second row (thus the third row) looks like the following is done by induction in subsubsection (\ref{dod}).
\begin{equation}
\xymatrix{
	[S^{2\alpha-2},\HPt]^G\ar[d]_{\cong}&[S^{3\alpha-2},\HPt]^G\ar[d]_{\cong}\ar[l]^{\aal}_{\cong}&[S^{4\alpha-2},\HPt]^G\ar[d]_{\cong}\ar[l]^{\aal}&[S^{5\alpha-2},\HPt]^G\ar[d]_{\cong}\ar[l]^{\aal}&\cdots\ar[l]\\
	\Z/2&\Z/2\ar[l]^1&\Z/4\ar[l]^1&\Z/2\ar[l]^2&\cdots\ar[l]
}\label{drb}
\end{equation}
Thus we get the differential
\[
d_3(b)=\frac{\overline{\ul}}{\overline{\al}}b^4.
\]

Similarly, we show $d_r(b^2)=0,r\geq 2$ by the following diagram. Note $b^2\in [\Sigma^{-2}G/{C_2}_+\wedge S^2,\HPt]^G$ maps to $[S(3\alpha)_+\wedge S^{-2},\HPt]^G$ and then can be lifted along the infinite sequence of isomorphisms. For degree reason, it cannot be hit by a differential. So $b^2\in E_{\infty}^{0,2}$.
{\scriptsize
\begin{equation}
	\xymatrixcolsep{1pc}
	\xymatrix{
		[S^{-3},\HPt]^G=0 &[S^{-3},\HPt]^G=0\ar[l]&[S^{-3},\HPt]^G=0\ar[l]&[S^{-3},\HPt]^G=0\ar[l]&\cdots\ar[l]\\
		[S^{2\alpha-3},\HPt]^G=0\ar[u]&[S^{3\alpha-3},\HPt]^G\ar[u]\ar[l]^-{\aal}&[S^{4\alpha-3},\HPt]^G\ar[u]\ar[l]^{\aal}_{\cong}&[S^{5\alpha-3},\HPt]^G\ar[u]\ar[l]^{\aal}_{\cong}&\cdots\ar[l]_-{\cong}\\
		[S(2\alpha)_+\wedge S^{-2},\HPt]^G\ar[u]^{\cong}&[S(3\alpha)_+\wedge S^{-2},\HPt]^G\ar[u]^{\cong}\ar[l]&[S(4\alpha)_+\wedge S^{-2},\HPt]^G\ar[u]^{\cong}\ar[l]&[S(5\alpha)_+\wedge S^{-2},\HPt]^G\ar[l]\ar[u]^{\cong}&\cdots\ar[l]\\
		[S^{-2},\HPt]^G=0\ar[u] &[S^{-2},\HPt]^G=0\ar[l]\ar[u]&[S^{-2},\HPt]^G=0\ar[l]\ar[u]&[S^{-2},\HPt]^G=0\ar[u]\ar[l]&\cdots\ar[l]
	}.
\end{equation}
}
That
\begin{equation}
	\xymatrix{
		[S^{2\alpha-3},\HPt]^G=0\ar[d]_{\cong}&[S^{3\alpha-3},\HPt]^G\ar[d]_{\cong}\ar[l]^{\aal}&[S^{4\alpha-3},\HPt]^G\ar[d]_{\cong}\ar[l]^{\aal}_{\cong}&[S^{5\alpha-3},\HPt]^G\ar[d]_{\cong}\ar[l]^{\aal}_{\cong}&\cdots\ar[l]_-{\cong}\\
		0&\Z/2\ar[l]&\Z/2\ar[l]_{\cong}&\Z/2\ar[l]_{\cong}&\cdots\ar[l]_{\cong}
	}\label{b^2}
\end{equation}
is an infinite sequence of isomorphism of $\Z/2$'s from the second term on is also proved by induction in subsubsection (\ref{dod}).

Now for $e_{\alpha}\in [G/{C_2}_+\wedge S^{\alpha-1},\HPt]^G$, its image in $[S(\alpha)_+\wedge S^{\alpha-1},\HPt]^G$ cannot be lifted to 
$[S(4\alpha)_+\wedge S^{\alpha-1},\HPt]^G$, which implies the differential 
\[
d_3(\eal)=b^3\frac{\overline{\ul}}{\overline{\al}}\eal.
\]
{\tiny
\begin{equation}
	\xymatrixcolsep{1pc}
	\xymatrix{
		[S^{\alpha-2},\HPt]^G=0 &[S^{\alpha-2},\HPt]^G=0\ar[l]&[S^{\alpha-2},\HPt]^G=0\ar[l]&[S^{\alpha-2},\HPt]^G=0\ar[l]&\cdots\ar[l]\\
		[S^{2\alpha-2},\HPt]^G\ar[u]&[S^{3\alpha-2},\HPt]^G\ar[u]\ar[l]^{\aal}_{\cong}&[S^{4\alpha-2},\HPt]^G\ar[u]\ar[l]^{\aal}&[S^{5\alpha-2},\HPt]^G\ar[u]\ar[l]^{\aal}&\cdots\ar[l]\\
		[S(\alpha)_+\wedge S^{\alpha-1},\HPt]^G\ar[u]^{\cong}&[S(2\alpha)_+\wedge S^{\alpha-1},\HPt]^G\ar[u]^{\cong}\ar[l]_{\cong}&[S(3\alpha)_+\wedge S^{\alpha-1},\HPt]^G\ar[u]^{\cong}\ar[l]&[S(4\alpha)_+\wedge S^{\alpha-1},\HPt]^G\ar[l]\ar[u]^{\cong}&\cdots\ar[l]\\
		[S^{\alpha-1},\HPt]^G=0\ar[u] &[S^{\alpha-1},\HPt]^G=0\ar[l]\ar[u]&[S^{\alpha-1},\HPt]^G=0\ar[l]\ar[u]&[S^{\alpha-1},\HPt]^G=0\ar[u]\ar[l]&\cdots\ar[l]
	}.
\end{equation}
}
The second row (thus the third row) is
\[
\xymatrix{
	[S^{2\alpha-2},\HPt]^G\ar[d]_{\cong}&[S^{3\alpha-2},\HPt]^G\ar[d]_{\cong}\ar[l]^{\aal}_{\cong}&[S^{4\alpha-2},\HPt]^G\ar[d]_{\cong}\ar[l]^{\aal}&[S^{5\alpha-2},\HPt]^G\ar[d]_{\cong}\ar[l]^{\aal}&\cdots\ar[l]\\
	\Z/2&\Z/2\ar[l]_1&\Z/4\ar[l]_1&\Z/2\ar[l]_2&\cdots\ar[l]
}
\]

For $\eal b$, we have
{\tiny
\begin{equation}
	\xymatrixcolsep{1pc}
	\xymatrix{
		[S^{\alpha-3},\HPt]^G=0 &[S^{\alpha-3},\HPt]^G=0\ar[l]&[S^{\alpha-3},\HPt]^G=0\ar[l]&[S^{\alpha-3},\HPt]^G=0\ar[l]&\cdots\ar[l]\\
		[S^{2\alpha-3},\HPt]^G=0\ar[u]&[S^{3\alpha-3},\HPt]^G\ar[u]\ar[l]^{\aal}&[S^{4\alpha-3},\HPt]^G\ar[u]\ar[l]^{\aal}_{\cong}&[S^{5\alpha-3},\HPt]^G\ar[u]\ar[l]^{\aal}_{\cong}&\cdots\ar[l]\\
		[S(\alpha)_+\wedge S^{\alpha-2},\HPt]^G\ar[u]^{\cong}&[S(2\alpha)_+\wedge S^{\alpha-2},\HPt]^G\ar[u]^{\cong}\ar[l]&[S(3\alpha)_+\wedge S^{\alpha-2},\HPt]^G\ar[u]^{\cong}\ar[l]&[S(4\alpha)_+\wedge S^{\alpha-2},\HPt]^G\ar[l]\ar[u]^{\cong}&\cdots\ar[l]\\
		[S^{\alpha-2},\HPt]^G=0\ar[u] &[S^{\alpha-2},\HPt]^G=0\ar[l]\ar[u]&[S^{\alpha-2},\HPt]^G=0\ar[l]\ar[u]&[S^{\alpha-2},\HPt]^G=0\ar[u]\ar[l]&\cdots\ar[l]
	}.
\end{equation}
}
The situation is similar to $b^2$. The two rows in the middle are the same as (\ref{b^2}). So we have $\eal b\in E_{\infty}^{\alpha-1,1}$ for degree reason.
\end{proof}

As a result, we have
\begin{Cor}
The $E_{\infty}$-page of (\ref{PholimSS}) is
\begin{equation}
	\begin{aligned}
		E_{\infty}^{\bigstar,*}=E_4^{\bigstar,*}&=\Z/2[\overline{\al}^{\pm},\overline{\ul},e_{2\alpha}^{\pm},b^2]\langle 1,\eal b\rangle/(b^3\overline{\ul}\eal,b^4\overline{\ul})\\
		&=\Z/2[\overline{\al}^{\pm},e_{2\alpha}^{\pm},b^2]\langle 1,\eal b\rangle\\
		&\oplus\Z/2[\overline{\al}^{\pm},e_{2\alpha}^{\pm}]\langle \overline{\ul}^{i\geq 1}\rangle\langle 1,\eal b,b^2\rangle.
	\end{aligned}
\end{equation}
\end{Cor}

By definition of $\overline{\al}=\res(\al),\overline{\ul}=\res(\ul),\res(\uta^{\pm})=e_{2\alpha}^{\mp}$, we have $\overline{\al}^{\pm},\overline{\ul},e_{2\alpha}^{\pm}$ converge to $\al,\ul,\uta^{\mp}\in\pi_{\bigstar}F(E\PP_+,\HPt)$ repectively. Now the element $\frac{\aal}{\uta}\in\pi_{\alpha-2}F(E\PP_+,\HPt)$ restricts to $0\in \pi_{\alpha-2}F(G/{C_2}_+,\HPt)=\pi_{\alpha-2}^{G/C_2}\HPt=0$. We have that $\eal b$ converges to $\frac{\aal}{\uta}$. The element $\frac{\atal}{\uta}\in\pi_{-2}F(E\PP_+,\HPt)$ restricts to $0\in [S(2\alpha)_+\wedge S^{-2},\HPt]^G=0$, thus $b^2$ converges to $\frac{\atal}{\uta}$.

Now we solve the extenison problems. The SS strongly convergences as said before. Denote $\pi_{\star}F(E\PP_+,\HPt)=R$ for simplicity, then we have $R\cong\lim_s R/F^s$ where 
\[
F^s=ker(\pi_{\star}F(E\PP_+,\HPt)\to \pi_{\star}F(S(s\alpha)_+,\HPt)),
\] 
and 
\[
F^{0}=ker(\pi_{\star}F(E\PP_+,\HPt)\to \pi_{\star}F(*,\HPt))=R.
\] 
The tower
\[
\xymatrix{
	0=\frac{R}{F^{0}}&\frac{R}{F^1}\ar[l]&\frac{R}{F^2}\ar[l]&\frac{R}{F^3}\ar[l]&\frac{R}{F^4}\ar[l]&\cdots\ar[l]&\lim_s\frac{R}{F^s}\cong R\ar[l]\\
	&\frac{F^{0}}{F^1}\ar[u]&\frac{F^{1}}{F^2}\ar[u]&\frac{F^{2}}{F^3}\ar[u]&\frac{F^{3}}{F^4}\ar[u]
}
\]
has $\frac{F^s}{F^{s+1}}\cong E_{\infty}^{\bigstar,s}$.

A careful examination of the degrees of elements show that the only possible extensions are
\[
0\leftarrow \Z/2[\al^{\pm},\uta^{\pm},\ul]\leftarrow \Z/4[\al^{\pm},\uta^{\pm},\ul]\xleftarrow{2} \Z/2[\al^{\pm},\uta^{\pm},\ul]\langle \ul\aal^2\rangle\leftarrow 0.
\]
By linearity, we only need to prove the following
\[
0\leftarrow \Z/2\langle \uta\al\rangle\leftarrow \Z/4\langle \uta\al\rangle\xleftarrow{2} \Z/2\langle \ul\aal^2\rangle\leftarrow 0.
\]
The group in the middle is $\Z/4$ because of the gold relation. The rest of the extension problems are trivial for degree reason. As a result, we have
\begin{Cor}
We have
\begin{equation}
	\pi_{\bigstar}(F(E\PP_+,\HPt))=R=\Z/4[\al^{\pm},\uta^{\pm},\ul,\aal]/(\atal\ul=2\al\uta,2\aal)
\end{equation}
\end{Cor}

\begin{Rem} Actually, the above spectral sequence can be understood as the $\aal$-Bockstein spectral sequence. Because if we apply Spanier-Whitehead dual to the cofiber sequence $S(k\alpha)_+\to S^0\xrightarrow{\aal^k} S^{k\alpha}$, we get $D(S(k\alpha)_+)\xleftarrow{}S^0\xleftarrow{\aal^k} S^{-k\alpha}$. Thus $D(S(k\alpha)_+)=C\aal^k$, where $C\aal^k$ means cofiber of the map $\aal^k$. Then the tower (\ref{tower}) becomes $*\xleftarrow{}C\aal\wedge \HPt\xleftarrow{}C\aal^2\wedge\HPt\xleftarrow{}C\aal^3\wedge\HPt\xleftarrow{}\cdots$. In the same way, the HFPSS can also be viewed as the $\al$-Bockstein spectral sequence.
\end{Rem}

Now we return to the $\PP$-Tate square (\ref{FTate2}). The homotopy group of the right upper corner is computed by \cite[Prop.3.18]{HHRa}. It has $\pi_{*}\HPf=\Z/2[\frac{\uta}{\atal}]$. Thus 
\[
\pi_{\bigstar}\HPf=\Z/2[\uta,\aal^{\pm},\al^{\pm}].
\]

Since $\widetilde{E\{e,C_2\}}\simeq S^{\infty\alpha}$, the map $E\to E\wedge\widetilde{E\{e,C_2\}}$ is localization at $\aal$ for any $G$-spectrum $E$. Thus
\[
\pi_{\bigstar}\widetilde{E\PP}\wedge F(E\PP_+,\HPt)=\pi_{\bigstar}F(E\PP_+,\HPt)[\aal^{\pm}]=\Z/2[\uta^{\pm},\aal^{\pm},\al^{\pm}].
\]
Note $\ul$ is $\aal$-torsion when working over $\Z/2$. Thus it is killed when inverting $\aal$. We can also notice that the map $\HPf=\widetilde{E\PP_+}\wedge \HPt\to \widetilde{E\PP_+}\wedge F(E\PP_+,\HPt)$ is localization at $\uta$. Some simple calculation in the second row gives us
\begin{align*}
	\pi_{\bigstar}E\PP_+\wedge\HPt&=\Z/4[\al^{\pm},\uta^{\pm}]\langle \ul^i\rangle_{i\geq 1}\\
	&\oplus\Z/2[\al^{\pm},\uta^{\pm}]\langle \ul^i\aal\rangle_{i\geq 1}\\
	&\oplus 2\Z/2[\al^{\pm},\uta^{\pm}]\\
	&\oplus\Z/2\langle \Si\aal^{-i}\rangle_{i\geq 1}[\al^{\pm},\uta^{\pm}].
\end{align*}
The $2\Z/2\langle 1\rangle$ means it is $\Z/2$ and the generator maps to $2\in \Z/4$ in $\pi_{\bigstar}F(E\PP_+,\HPt)$. Others follow from linearity. Sometimes we may write $2\Z/2\langle 1\rangle=\Z/2\langle 2\rangle$.

Now consider the following exact sequence
\[
0\to K^{\Phi_2}\to \pi_{\bigstar}\HPf\xrightarrow{\delta}\pi_{\bigstar-1}E\PP_+\wedge\HPt\to C^{\Phi_2}\to 0.
\]
Here $K^{\Phi_2},C^{\Phi_2}$ denotes kernel and cokernel respectively. The only nontrivial connecting maps are
\[
\Z/2\langle \aal^{-i}\rangle_{i\geq 1}[\uta,\al^{\pm}]\underset{\cong}{\xrightarrow{\delta}}\Z/2\langle \Si\aal^{-i}\rangle_{i\geq 1}[\uta,\al^{\pm}]
\]
As a result, we are left with the extension problems
\begin{equation*}
0 \rightarrow C^{\Phi_2} \rightarrow\pi_{\bigstar}\HPt \rightarrow K^{\Phi_2} \rightarrow 0.
\end{equation*}
For degree reason, the only possible non-trivial extension is the following
\begin{align*}
	0 \rightarrow 2\Z/2[\al^{\pm},\uta]
	\rightarrow \pi_{\bigstar}\HPt \rightarrow \Z/2[\al^{\pm},\uta] \rightarrow 0.\\
\end{align*}
By linearity, we consider the above extension in degree 0. Since every spectrum in the $\PP$-Tate diagram (\ref{FTate}) has $\pi_1=0,\pi_{-1}=0$, the 5-lemma tells us the vertical map in the middle induced by $\HPt\to F(E\PP_+,\HPt)$ is an isomorphism
\[
\xymatrix{
	0\ar[r] &\Z/2\langle 2\rangle\ar[r]\ar[d]_{\cong} &\pi_0(\HPt)\ar[r]\ar[d] &\Z/2\langle 1\rangle\ar[r]\ar[d]_{\cong} &0\\
	0\ar[r] &\Z/2\langle 2\rangle\ar[r]^-2&\Z/4\langle 1\rangle\ar[r]^-1 &\Z/2\langle 1\rangle\ar[r]&0.
}
\]
Thus, we conclude
\begin{Thm} The top level of the $RO(C_4)$-graded Green functor $\underline{\pi_{\bigstar}}\HPt$ is
\begin{equation*}
	\begin{split}
		\HPt_{\bigstar}&=\Z/4[\al^{\pm},\uta,\ul,\aal]/(2\al\uta=\atal\ul,2\aal)\\
		&\oplus\Z/4[\al^{\pm}]\langle \uta^{-i}\ul^j\rangle_{i,j\geq 1}\langle 1,\aal\rangle/2\aal\\
		&\oplus\Z/2\langle \Si\aal^{-i}\uta^{-j}\rangle_{i,j\geq 1}[\al^{\pm}]\\
		&\oplus\Z/2[\al^{\pm}]\langle 2\uta^{-i}\rangle_{i\geq 1}.
	\end{split}
\end{equation*}
Note that by our notational convention, $\HPt_{\bigstar}=\underline{\pi_{\bigstar}}\HPt(G/G)$. 
\end{Thm}

\begin{Rem}
Here the classes $\Z/2\langle \Si\uta^{-i}\aal^{-j}\rangle_{i,j\geq 1}$, unlike before, do not come from any part of our main Tate square (\ref{Tate}) as we can find $\aal^{-1}$ from nowhere. It should be understood as a class $\Z/2\langle \Si\uta^{-1}\aal^{-1}\rangle\in \pi_{3\alpha-3}\HPt$ that is infinitely divisible by $\aal$ and $\uta$. From the next subsection we will see that this class is the image under $H\to \HPt$ of a class $\tre\in\pi_{3\alpha-3}H$.\\
\end{Rem}
For clarity, we will denote $\Z/2\langle \Si\uta^{-i}\aal^{-j}\rangle_{i,j\geq 1}=\Z/2\Tre$. Thus we get
\begin{equation}
	\begin{split}
		\HPt_{\bigstar}&=\Z/4[\al^{\pm},\uta,\ul,\aal]/(2\al\uta=\atal\ul,2\aal)\\
		&\oplus\Z/4[\al^{\pm}]\langle \uta^{-i}\ul^j\rangle_{i,j\geq 1}\langle 1,\aal\rangle/2\aal\\
		&\oplus\Z/2\Tre[\al^{\pm}]\\
		&\oplus\Z/2[\al^{\pm}]\langle 2\uta^{-i}\rangle_{i\geq 1}.
	\end{split}\label{HPt}
\end{equation}	

\subsubsection{Determination of the differentials}\label{dod} To determine the differentials in the last subsubsection, we need to know the sequences of groups
\begin{equation}
	\pi_{2\alpha-2}\HPt\xleftarrow{\aal} \pi_{3\alpha-2}\HPt\xleftarrow{\aal} \pi_{4\alpha-2}\HPt\xleftarrow{\aal} \pi_{5\alpha-2}\HPt\xleftarrow{\aal}\cdots\label{2alphatower}
\end{equation}

\begin{equation}
	\pi_{2\alpha-3}\HPt\xleftarrow{\aal} \pi_{3\alpha-3}\HPt\xleftarrow{\aal} \pi_{4\alpha-3}\HPt\xleftarrow{\aal} \pi_{5\alpha-3}\HPt\xleftarrow{\aal}\cdots\label{3alphatower}
\end{equation}

This is easily done using induction based on the following cofiber sequences
\[
G/{C_2}_+\wedge S^{k\alpha}\xrightarrow{res} S^{k\alpha}\to S^{(k+1)\alpha}
\]
obtained from smashing $G/{C_2}_+\xrightarrow{res} S^0\to S^{\alpha}$ with $S^{k\alpha}$. The map $res$ is induced from the only $G$-map $G/C_2\to *$ which will become restriction upon applying $[-,\Sigma^n\HPt]^G$. Since we only care about the $G/G$-level Mackey functor, we only need to compute
\[
\xymatrix{
	\pi_{*+k\alpha}\HPt(G/G) \ar@/_/[d]_{\res}\\
	\pi_{*+k\alpha}\HPt(G/C_2)\ar@/_/[u]_{\tr}
}
\]
for $k\geq 1$. Actually, since $\HPt$ is concentrated over $C_2$, i.e., $i^*_e\HPt\simeq *$, the bottom level is always 0, i.e., this is a $C_2'$ computation. And we have $i^*_{C_2}\HPt\simeq H\wedge \widetilde{EC_2}\simeq H[a_{\sigma}^{-1}]$.

To get started, in the cofiber sequence
\[
H_h\to H\to \HPt,
\]
using results from (\ref{H_h}) and the definition of $H$, we have
\begin{align*}
	H_h{}_*&=4\Z\langle 1\rangle\\
	&\oplus \Z/4\langle \Si(\frac{\ul}{\al})^i\rangle_{i\geq 1},\\
	H_*&=\Z\langle 1\rangle.
\end{align*}
As a result, we get $\HPt_*=\Z/4[\frac{\ul}{\al}]$. Thus
\[
\pi_{*}\HPt=
\xymatrix{
	\Z/4[\frac{\ul}{\al}] \ar@/_/[d]_1\\
	\Z/2[\frac{u_{2\sigma}}{\ats}]\ar@/_/[u]_2\ar@/_/[d]_0\\
	0\ar@/_/[u]_0.
}.
\]
$\tr=2$ is derived from $\res=1$ and the cohomological condition $\tr\circ \res=|G/C_2|=2$. This condition holds because $\HPt$ is a $H$-algebra thus its homotopy is a $\uZ$-module \cite[Prop 16.3]{TW95}.

\begin{Rem}
When we are dealing with a spectrum $\HPt$ which is much simper than $H$ itself, cellular induction could be a third method to compute its $RO(G)$-graded homotopy. The author actually also computed $\HPt_{\bigstar}$ using this induction method, but the $\PP$-Tate square is a cleaner one and it only require knowledge of $\HPt_{k\alpha-2}, \HPt_{k\alpha-3},k\geq 2$.
\end{Rem}

From the following two cofiber sequences (we apply ${G/C_2}_+\wedge -$ to the first one to get the second one, and they correspond to the computation of $\pi_{*+k\alpha}^{G/G}\HPt$ and $\pi_{*+k\alpha}^{G/C_2}\HPt$ respectively)
\begin{equation}
\begin{split}
\xymatrix{
	\ar[r]&G/{C_2}_+\ar[r]^-{res} &S^0\ar[r] &S^{\alpha}\ar[r] &G/{C_2}_+\wedge S^1\ar[r] &S^1\ar[r]&\cdots,\\
	\ar[r]&2G/{C_2}_+\ar[r]^-{\nabla} &G/{C_2}_+\ar[r] &G/{C_2}_+\wedge S^{\alpha}\ar[r] &2G/{C_2}_+\wedge S^1\ar[r]^-{\nabla} &S^1\ar[r]&\cdots,
}
\end{split}\label{induction}
\end{equation}
we know
\begin{equation*}
   \pi_{*+\alpha}\HPt=
\xymatrix{
	\langle 2\aal^{-1}\rangle\Z/2[\frac{\ul}{\al}] \ar@/_/[d]_0\\
	\langle \eal\rangle\Z/2[\frac{u_{2\sigma}}{\ats}]\ar@/_/[u]_0\ar@/_/[d]_0\\
    0\ar@/_/[u]_0
}, \label{1alpha}
\end{equation*}
where restricitons and transfers are $0$ for degree reason. The computation is as follows. In (\ref{induction}), 
\[
2G/{C_2}=G/{C_2}\coprod G/{C_2}.
\] 
The folding map $2G/{C_2}\xrightarrow{\nabla}G/{C_2}$ is identified with the map 
\[
G/{C_2}\times G/{C_2}\xrightarrow{1\times res}G/{C_2}\times *
\] 
through 
\[
G/{C_2}\times G/{C_2}\cong \{(1,1),(\bar{\gamma},\bar{\gamma})\}\coprod \{(1,\bar{\gamma}),(\bar{\gamma},1)\}.
\] 
It becomes diagonal upon applying $[-,\Sigma^n\HPt]^G,n\in\Z$.

Applying $[-,\Sigma^{-*}\wedge\HPt]^G$ to the first cofiber sequence in (\ref{induction}), we get the exact sequence
\begin{equation*}
    \cdots\leftarrow\pi^{G/C_2}_*\xleftarrow{\res}\pi_*\xleftarrow{\aal}\pi_{\alpha+*}\leftarrow{}\pi^{G/C_2}_{*+1}\xleftarrow{\res}\pi_{*+1}\leftarrow\cdots.
\end{equation*}
The non-trivial restrictions are the surjections $\Z/4\twoheadrightarrow \Z/2$ when $*=2k,k\geq 0$. The non-tirival classes all come from kernels of the restricitons, which are $2\cdot (\frac{\ul}{\al})^i,i\geq 0$. Lifting along the map $\aal$ gives us the $G/G$-level of the graded Mackey functor $\underline{\pi}_{*+\alpha}\HPt$.

Applying $[-,\Sigma^{-*}\wedge\HPt]^G$ to the second cofiber sequence in (\ref{induction}), we get the exact sequence
\begin{equation*}
    \cdots\leftarrow 2\pi^{G/C_2}_{*}\xleftarrow{\Delta}\pi^{G/C_2}_{*}\xleftarrow{\aal}\pi^{G/C_2}_{\alpha+*}\leftarrow 2\pi^{G/C_2}_{*+1}\xleftarrow{\Delta}\pi^{G/C_2}_{*+1}\leftarrow\cdots.
\end{equation*}
Here $2\pi^{G/C_2}_{*}=\pi^{G/C_2}_{*}\oplus \pi^{G/C_2}_{*}$ and $\Delta$ is the diagonal map. $\Delta$ is always injective, and in the non-trivial cases, we have $\Delta:\Z/2\to \Z/2\oplus\Z/2$. All the non-trivial classes come from cokernels of $\Delta$. Since the diagonals in $\Z/2[\frac{\uts}{\ats}]\oplus \Z/2[\frac{\uts}{\ats}]$ are killed, we can choose either the left or right copy of $\Z/2$ to get 
\[
Coker(\Delta)_{*+1}=\Z/2[\frac{\uts}{\ats}]\subset 2\pi^{G/C_2}_{*+1}, *=-1+2k, k\geq 0.
\]
Since $\eal\cdot \pi^{G/C_2}_{*+1}=\pi^{G/C_2}_{\alpha+*}$, we get the $G/C_2$-level of the graded Mackey functor $\underline{\pi}_{*+\alpha}\HPt$. $\res,\tr$ are $0$ for degree reason: the $G/G$-level is concentrated in even gradings of the trivial representation, while the $G/C_2$-level is concentrated in odd gradings of the trivial representation.

For $\underline{\pi}_{*+2\alpha}\HPt$, smash $S^{\alpha}$ with (\ref{induction}) and applying $[-,\Sigma^n\HPt]^G,n\in\Z$, we get
\begin{equation*}
   \underline{\pi}_{*+2\alpha}\HPt=
\xymatrix{
	\Z/2\langle2\uta^{-1}\rangle\ar@/_/[d]_0&&\Z/4\langle\uta^{-1} (\frac{\ul}{\al})^i\rangle_{i\geq 1} \ar@/_/[d]_1\\
	\Z/2\langle\etal\rangle\ar@/_/[d]_0\ar@/_/[u]_1&\oplus &\Z/2\langle\etal\cdot(\frac{\uts}{\ats})^i\rangle_{i\geq 1}\ar@/_/[u]_2\ar@/_/[d]_0\\
	0\ar@/_/[u]_0&&0\ar@/_/[u]_0.
}, \label{2alpha}
\end{equation*}
The Mackey structure of the first summand is determined by degree reason and that the top class is killed by $\aal$ (\ref{lemrestr}). It can also be deduced by noting that the map $S^{2\alpha}\to G/{C_2}_+\wedge S^{1+\alpha}\cong G/{C_2}_+\wedge S^{2\alpha}$ becomes $\tr$ after applying a contravariant functor. The generator on the top is determined by Frobenius reciprocity: we get $\tr(\etal)=2\uta^{-1}$ by noticing $\tr(\etal)\cdot \uta=\tr(\etal\cdot\res(\uta))=\tr(1)=2$. The second summand follows from $(\frac{\ul}{\al})$-linearity once we know the following extension
\begin{equation*}
\begin{aligned}
    0\leftarrow{}\pi_{\alpha}=\Z/2\langle2\aal^{-1}\rangle\xleftarrow{\aal}\pi_{2\alpha}=\Z/4\langle\uta^{-1}\cdot \frac{\ul}{\al}\rangle\xleftarrow{\tr}\pi^{G/C_2}_{2\alpha}=\Z/2\langle\etal\cdot(\frac{\uts}{\ats})\rangle\leftarrow 0.
\end{aligned}
\end{equation*}
The gold relation tells us that the middle group is $\Z/4$: Firstly, we have $2\aal^{-2}=\uta^{-1}\frac{\ul}{\al}$, where the former class is a lift of $2\aal^{-1}$ along $\aal$. This is true since if $\aal^2\cdot\uta^{-1}\cdot \frac{\ul}{\al}=0\in\pi_0\HPt$ rather than $2$, it will lead to a contradiction once we map to $H^t_{\bigstar}$. Secondly, $\tr(\etal\cdot\frac{\uts}{\ats})=2\cdot\uta^{-1}(\frac{\ul}{\al})$ by Frobenious reciprocity. Thus we deduce that the middle group is $\Z/4$. $\res=1$ follows from the cohomological condition.

Smashing $S^{2\alpha}$ with (\ref{induction}) and applying $[-,\Sigma^n\HPt]^G,n\in\Z$ gives us
\begin{equation*}
   \underline{\pi}_{*+3\alpha}\HPt=
\xymatrix@C0.5em{
	\Z/2\langle\tre\rangle\ar@/_/[d]_0&&\Z/2\langle2\aal^{-1}\uta^{-1}\rangle \ar@/_/[d]_0&&\Z/2\langle2\aal^{-1}\uta^{-1}(\frac{\ul}{\al})^i\rangle_{i\geq 1} \ar@/_/[d]_0\\
	\Z/2\langle\ethal\rangle\ar@/_/[d]_0\ar@/_/[u]_1&\oplus &0\ar@/_/[u]_0\ar@/_/[d]_0&\oplus&0\ar@/_/[u]_0\ar@/_/[d]_0\\
	0\ar@/_/[u]_0&&0\ar@/_/[u]_0&&0\ar@/_/[u]_0.
} \label{3alpha}
\end{equation*}

\begin{Rem}
This is a good way to notice the existence of the class $\tre\in \HPt_{3\alpha-3}$ which lifts to $H_{3\alpha-3}$ since $({H_h})_{3\alpha-3}=({H_h})_{3\alpha-4}=0$.
\end{Rem}

Keep going and solving extensions by the gold relation as above gives us:
\begin{equation*}
   \underline{\pi}_{*+4\alpha}\HPt=
\xymatrix@C0.5em{
	\Z/2\langle\tr(e_{4\alpha})\rangle\ar@/_/[d]_0&&\Z/2\langle\aal^{-1}\tre\rangle \ar@/_/[d]_0&&\Z/4\langle(\frac{\ul}{\al})^i\uta^{-2}\rangle_{i\geq 1} \ar@/_/[d]_1\\
	\Z/2\langle e_{4\alpha}\rangle\ar@/_/[d]_0\ar@/_/[u]_1&\oplus &0\ar@/_/[u]_0\ar@/_/[d]_0&\oplus&\Z/2\langle (\frac{\uts}{\ats})^i\etal^2\rangle_{i\geq 1}\ar@/_/[u]_2\ar@/_/[d]_0\\
	0\ar@/_/[u]_0&&0\ar@/_/[u]_0&&0\ar@/_/[u]_0.
} \label{4alpha}
\end{equation*}

and
\begin{equation*}
   \underline{\pi}_{*+5\alpha}\HPt=
\xymatrix@C0.5em{
	\Z/2\langle \tre\uta^{-1}\rangle\ar@/_/[d]_0&&\Z/2\langle2\aal^{-1}\uta^{-2}\rangle[\frac{\ul}{\al}] \ar@/_/[d]_0&&\Z/2\langle\tre\aal^{-2}\rangle\ar@/_/[d]_0\\
	\Z/2\langle e_{5\alpha}\rangle\ar@/_/[d]_0\ar@/_/[u]_1&\oplus &0\ar@/_/[u]_0\ar@/_/[d]_0&\oplus&\Z/2\langle e_{5\alpha}\cdot\frac{\uts}{\ats}\rangle\ar@/_/[u]_0\ar@/_/[d]_0\\
	0\ar@/_/[u]_0&&0\ar@/_/[u]_0&&0\ar@/_/[u]_0.
} \label{5alpha}
\end{equation*}

Here $\tr(5\alpha)=\tre\uta^{-1}$ follows from Frobenious reciprocity. In the last summand, $\res=0$ is deduced from noticing $\pi_{6\alpha-3}\HPt=\Z/2\langle \tre\aal^{-3}\rangle$ and use (\ref{lemrestr}). Actually $\tre$ is infinitely $\aal$-divisible, so all the restrictions involving it are $0$. Again by (\ref{lemrestr}), $\tr=0$.

We can keep going and use Spanier-Whitehead duality to (\ref{induction}) to compute $\pi_{*\alpha+*}\HPt$. But to understand (\ref{2alphatower}) and (\ref{3alphatower}), knowing $\pi_{i\alpha-2}\HPt$ and $\pi_{i\alpha-3}\HPt$ for $i\geq 2$ is enough. The $\aal$-tower in (\ref{2alphatower}) is
\begin{equation*}
    \Z/2\langle 2\uta^{-1}\rangle\underset{1}{\xleftarrow{\aal}}\Z/2\langle 2\aal^{-1}\uta^{-1}\rangle\underset{1}{\xleftarrow{\aal}}\Z/4\langle (\frac{\ul}{\al})\uta^{-2}\rangle\underset{2}{\xleftarrow{\aal}}\Z/2\langle 2\aal^{-1}(\frac{\ul}{\al})\uta^{-2}\rangle\xleftarrow{\aal}\cdots.
\end{equation*}
The $\aal$-tower in (\ref{3alphatower}) is
\begin{equation*}
    0\underset{}{\xleftarrow{\aal}}\Z/2\langle \tre\rangle\underset{1}{\xleftarrow{\aal}}\Z/2\langle \aal^{-1}\tre\rangle\underset{1}{\xleftarrow{\aal}}\Z/2\langle \aal^{-2}\tre\rangle\xleftarrow{\aal}\cdots.
\end{equation*}
Depending on the parity of the number of $\alpha$ in the grading, the second tower is deduced either for degree reason or by noticing $\res=1$ on $\pi_{2k\alpha-2}\HPt=\Z/4\langle (\frac{\ul}{\al})^{k-1}\uta^{-k}\rangle,k\geq 2$.\\

\section{Remaining computations in the main Tate square}
Now we will finish the computation in the main Tate square (\ref{Tate}). The essential step is to understand the connecting homomorphism
\[
0\to K\to \HPt_{\bigstar}\xrightarrow{\delta} {H_h}_{\bigstar-1}\to C\to 0.\label{connecting}
\]

For convenience, we decompose $\HPt_{\bigstar}$ into two parts
\[
\HPt_{\bigstar}=\HPt_{\leq 0}\oplus \HPt_{\geq 1}
\]
where
\begin{align}
\HPt_{\leq 0}&=\Z/4[\al^{\pm},\aal,\ul,\uta]/(4\al,2\aal,\atal\ul=2\al\uta)\\
\HPt_{\geq 1}&=\Z/2\Tre[\al^{\pm}]\label{1}\\
&\oplus\Z/2[\al^{\pm}]\langle \uta^{-i}\ul^j\rangle_{i,j\geq 1}\langle \aal\rangle\label{2}\\
&\oplus\Z/2\langle 2\uta^{-i}\rangle_{i\geq 1}[\al^{\pm}]\label{3}\\
&\oplus\Z/4\langle \uta^{-k}\ul^i\rangle_{i,k\geq 1}[\al^{\pm}].\label{4}
\end{align}
Here $\leq 0,\geq 1$ means the coefficients before $\alpha$. Accordingly, we compute the connecting homomorphism in two steps: we have the following two exact sequences

\begin{align}
 0\to K_1\to \HPt_{\leq 0}\xrightarrow{\delta} {H_h}_{\bigstar-1}\to C_1\to 0.\label{seq1}
\end{align}

\begin{align}
 0\to K_2\to \HPt_{\geq 1}\xrightarrow{\delta} C_1\to C\to 0.\label{seq2}
\end{align}
Obviouisly, $K=K_1\oplus K_2$. For the first sequence (\ref{seq1}) additively,
\begin{align*}
 \HPt_{\leq 0}&=\Z/4[\al^{\pm},\uta,\ul]\langle 1,\aal\rangle/2\aal\\
 &\oplus \Z/2[\al^{\pm},\uta]\langle \aal^i\rangle_{i\geq 2}.
\end{align*}
There are obvious isomorphisms:

\begin{align*}
 &\Z/2\langle \frac{\aal}{\al^i}\rangle_{i\geq 1}[\uta,\ul]\xr{\quad\delta\quad}{\cong}\Z/2\langle \Si\frac{\aal}{\al^i}\rangle_{i\geq 1}[\uta,\ul]\\
 &\Z/4\langle \frac{1}{\al^i}\rangle_{i\geq 1}[\uta,\ul]\xr{\quad\delta\quad}{\cong}\Z/4\langle \Si\frac{1}{\al^i}\rangle_{i\geq 1}[\uta,\ul]\\
\end{align*}

Then we look at the elements with factor $\aal^i,i
\geq 2$. Note that $\delta$ is the composite of the three maps in the Tate square:

\[
 \xymatrix{
 \HPt\ar[r]^{\delta}\ar[d] &\Sigma H_h\\
 H\uZ^t\ar[r] &\Sigma H_h\ar[u]_{\cong}
 }.\label{connecting square}
\]
$\delta(\frac{\atal}{\al})=2\cdot\Si\frac{\uta}{\ul}=0$ in ${H_h}_{\bigstar}$, which can be computed using the gold relation in $H\uZ^t_{\bigstar}$. Thus $\delta(\frac{a_{\alpha}^i}{\al})=2\cdot\Si\frac{\uta\aal^{i-2}}{\ul}=0,i\geq 2$, which implies $\frac{a_{\alpha}^i}{\al}[\uta,\al]\in K,i\geq 2$.\\ 

On the other hand, $\delta(\frac{\atal}{\al^2})=2\cdot \Si\frac{\uta}{\al\ul}\neq 0$ in ${H_h}_{\bigstar}$. We have the following sequence:
\[
 0\to\Z/2\langle \frac{\atal}{\al^2}\rangle[\uta,\frac{1}{\al}]\xr{\delta}{2}\Z/4\langle \Si\frac{\uta}{\al\ul}\rangle[\uta,\frac{1}{\al}]\to \Z/2\langle \Si\frac{\uta}{\al\ul}\rangle[\uta,\frac{1}{\al}]\to 0,
\]
providing us $\Z/2\langle \Si\frac{\uta}{\al\ul}\rangle[\uta,\frac{1}{\al}]\in C_1$.

Since all elements in $\HPt_{\bigstar}$ are torsion, the elements $4\Z[\uta^{\pm},\ul^{\pm}]$ survives and gives elements in $C_1$. Thus we get
\begin{align*}
 K_1&=\Z/4[\al,\uta,\ul]\langle 1,\aal\rangle/2\aal\\
 &\oplus\Z/2\langle \aal^i\rangle_{i\geq 2}[\uta,\al]\\
 &\oplus\Z/2\langle \frac{\aal^j}{\al}\rangle_{j\geq 2}[\uta]\\
 &\oplus\Z/2\langle \aal^i\al^{-j}\rangle_{i\geq 3,j\geq 2}[\uta].
\end{align*}

\begin{align*}
 C_1&=4\Z[\uta^{\pm},\ul^{\pm}]\\
 &\oplus\Z/2\langle \Si\frac{\aal}{\al^i\uta^j}\rangle_{i,j\geq 1}[\ul]\\
 &\oplus\Z/2\langle \Si\frac{\aal}{\al^i\ul^j}\rangle_{i,j\geq 1}[\uta^{\pm}]\\
 &\oplus\Z/4\langle \Si\frac{1}{\al^i\uta^j}\rangle_{i,j\geq 1}[\ul]\\
 &\oplus\Z/4\langle \Si\frac{1}{\al^i\ul^j}\rangle_{i\geq, j\geq 2}[\uta^{\pm}]\\
 &\oplus\Z/4\langle \Si\frac{1}{\al^i\ul}\rangle_{i\geq 1}[\uta^{-1}]\\
 &\oplus\Z/2\langle \Si\frac{1}{\al^i\ul}\rangle_{i\geq 1}\langle \uta^j\rangle_{j\geq 1}.
\end{align*}

Now we consider the sequence (\ref{seq2}). For degree reason, the elements in (\ref{1}) do not hit any element under $\delta$. For (\ref{2}), we have the following isomorphism:
\begin{align*}
\Z/2\langle \frac{\aal\ul^i}{\uta^j\al^k}\rangle_{i,j,k\geq 1}\xr{\quad\delta\quad}{\cong}\Z/2\langle \Si\frac{\aal\ul^i}{\uta^j\al^k}\rangle_{i,j,k\geq 1}
\end{align*}
providing us $\Z/2\langle \frac{\aal\ul^j}{\uta^i}\rangle[\al]\in K_2$ and $\Z/2\langle \Si\frac{\aal}{\al^i\uta^j}\rangle_{i,j\geq 1}\in C$.

With the obvious name comparison, we have the follwing exact sequence and isomorphism:
\begin{equation}
 0\to \Z/2\langle 2\uta^{-i}\al^{-j}\rangle_{i,j\geq 1}\xr{2}{\delta}\Z/4\langle \Si\uta^{-i}\al^{-j}\rangle_{i,j\geq 1}\to \Z/2\langle \Si\uta^{-i}\al^{-j}\rangle_{i,j\geq 1}\to 0, 
\end{equation}
and
\begin{equation}
 \Z/4\langle \uta^{-k}\ul^i\al^{-j}\rangle_{i,j,k\geq 1}\xr{\quad\delta\quad}{\cong}\Z/4\langle \Si\uta^{-k}\ul^i\al^{-j}\rangle_{i,j,k\geq 1}.
\end{equation}

Take every thing into account, we get
\begin{align*}
 K_2&=\Z/2\Tre[\aal^{-1}][\al^{\pm}]\\
 &\oplus \Z/2\langle \frac{\aal\ul^i}{\uta^j}\rangle_{i,j\geq 1}[\al]\\
 &\oplus\Z/2\langle 2\uta^{-i}\rangle_{i\geq 1}[\al]\\
 &\oplus \Z/4\langle \uta^{-k}\ul^{i}\rangle_{i,k\geq 1}[\al].
\end{align*}

\begin{align*}
 C&=4\Z[\uta^{\pm},\ul^{\pm}]\\
 &\oplus \Z/2\langle \Si\frac{\aal}{\al^i\ul^k}\rangle_{i,k\geq 1}[\uta^{\pm}]\\
 &\oplus \Z/4\langle \Si\frac{1}{\al^i\ul^j}\rangle_{i\geq 1,j\geq 2}[\uta^{\pm}]\\
 &\oplus \Z/4\langle \Si\frac{1}{\al^i\ul}\rangle_{i\geq 1}[\uta^{-1}]\\
 &\oplus \Z/2\langle \Si\frac{\uta^j}{\al^i\ul}\rangle_{i,j\geq 1}\\
 &\oplus \Z/2\langle \Si\frac{\aal}{\al^i\uta^j}\rangle_{i,j\geq 1}\\
 &\oplus \Z/2\langle \Si\frac{1}{\al^i\uta^j}\rangle_{i,j\geq 1}.
\end{align*}
Now we are only left with the extension problems:
\begin{equation}
 0\to C\to H\uZ_{\bigstar}\to K\to 0\label{ext}
\end{equation}
with 
\begin{align*}
 K=K_1\oplus K_2.
\end{align*}

For classes of the form $\theta\uta^i\ul^j$ for $i,j\in \Z$ and $\theta=1,2,4$, we have the following extensions except the exotic case (\ref{exotic}) that will be discussed below:

\begin{Prop}
The only extensions in (\ref{ext}) are the ones listed below and the exotic ones (\ref{exotic}).
\begin{equation}
\begin{split}
 &0\to 4\Z[\uta,\ul]\xrightarrow{4}\Z[\uta,\ul]\to\Z/4[\uta,\ul]\to 0,\\
 &0\to 4\Z\langle \uta^{-1}\rangle_{i\geq}\xrightarrow{2}\Z\langle 2\uta^{-i}\rangle_{i\geq 1}\to \Z/2\langle 2\uta^{-i}\rangle_{i\geq 1}\to 0,\\
 &0\to 4\Z\langle \uta^{-k}\ul^i\rangle_{i,k\geq 1}\xrightarrow{4}\Z\langle \uta^{-k}\ul^i\rangle_{i,k\geq 1}\to \Z/4\langle \uta^{-k}\ul^i\rangle_{i,k\geq 1}\to 0,\\
 &0\to \Z\langle 4\ul^{-j}\rangle_{j\geq 2}[\uta^{-1}]\xrightarrow{1}\Z\langle 4\ul^{-j}\rangle_{j\geq 1}[\uta^{-1}]\to 0\to 0,\\
 &0\to 4\Z\langle 2\uta^k\ul^{-1}\rangle_{k\geq 1}\xrightarrow{2}\Z\langle 2\uta^k\ul^{-1}\rangle_{k\geq 1}\to \Z/2\langle \frac{\atal}{\al}\rangle[\uta]\to 0,\\
 &0\to \Z\langle 4(\frac{\uta}{\ul})^k\rangle_{k\geq 1}\langle \ul^{-i}\rangle_{i\geq 1}\xrightarrow{}\Z\langle 4(\frac{\uta}{\ul})^k\rangle_{k\geq 1}\langle \ul^{-i}\rangle_{i\geq 1}\to 0 \to 0,\label{uext2}
\end{split}
\end{equation}
\end{Prop}

\begin{proof}
Once we proved the middle group are $\Z$'s, the name of the generators are pretty straightforward by mapping it to $H^h$. To see the middle groups are $\Z$'s (we only need to care about the cases in Line 1,2,3,5), let us consider the Tate square from another perspective:
\begin{equation}
	\xymatrix{
	F\ar[r]^{\simeq}\ar[d] &F\ar[d]\\
	H\ar[r]\ar[d] &\HPt\ar[d]\\
	H^h\ar[r] &H^t.
}
\end{equation}
Here the two $F$'s are the fibers of the obvious vertical maps. We know that for a class in ${H^h}_{\bigstar}$ which is in the subring $\Z[\uta^{\pm},\ul^{\pm}]$, it generate a $\Z$. Since the only subgroups of $\Z$ are $d\Z,d\in\Z_{\geq 0}$, we only need to prove that the Borel completion map $H\to H^h$ induce injections in the degrees that appear in (\ref{uext2}). Equivalently, it suffices to prove $\pi_VF=0$ in the degrees that appear in (\ref{uext2}).

Just by looking at the names, we have the following exact sequences telling us Line 1 and Line 2 in (\ref{uext2}).
\begin{equation}
	\xymatrix{
		F_{V}\ar[r]^{\simeq}\ar[d] &0\ar[d]\\
		\HPt_{V}\ar[r]^-{\cong}\ar[d]&\Z/4[\uta,\ul]\ar[d]^1\\
		H^t_{V}\ar[r]^-{\cong} &\Z/4[\uta,\ul].
	}
\end{equation}

\begin{equation}
	\xymatrix{
		F_{V}\ar[r]^{\simeq}\ar[d] &0\ar[d]\\
		\HPt_{V}\ar[r]^-{\cong}\ar[d]&\Z/2\langle 2
		\uta^{-i}\rangle_{i\geq 1}\ar[d]^2\\
		H^t_{V}\ar[r]^-{\cong} &\Z/4\langle \uta^{-i}\rangle_{i\geq 1}.
	}
\end{equation}

Note here we abused the notation a little bit. $V$ denote the corresponding degrees of the elements on the right. For Line 3, if the any of the classes $4\Z\langle \uta^{-k}\ul^i\rangle_{i,k\geq 1}$ lives in degree $V$, then we have $\pi_V\HPt=\pi_{V+1}H^t=0$, thus $\pi_VF=0$.

For Line 5, we have the gold relation in $H^t_{\bigstar}$
\[
\frac{\atal}{\al}=2\frac{\uta}{\ul}.
\]
This implies that we have the short exact sequences
\begin{equation}
	\begin{split}
	\xymatrix{
		F_V\ar[r]^{\simeq}\ar[d] &0\ar[d]\\
		\HPt_V\ar[r]^-{\cong}\ar[d]&\Z/2\langle \frac{\atal}{\al}\rangle[\uta]\ar[d]^2\\
		H^t_V\ar[r]^-{\cong} &\Z/4\langle \frac{\uta}{\ul}\rangle[\uta].
	}
   \end{split}\label{uext3}
\end{equation}
Thus we have proved that the middle groups in (\ref{uext2}) are $\Z$'s.
\end{proof}

\begin{Rem}
Note that \cite[Prop 3.3]{HHRb} proved that the middle group are $\Z$'s in the case $\pi_{|V|-V}H$ and $\pi_{V-|V|}H$, where $V$ is an actual representation of $G$ and $V^G=0$. They did not consider the cases of virtual representations of underlying degree $0$ like $\pi_{2\alpha-\lambda}H=\Z\langle \uta^{-1}\ul\rangle$, $\pi_{\lambda-2\alpha}H=\Z\langle 2\uta\ul^{-1}\rangle$, etc.
\end{Rem}

In the next proposition we prove the exotic extensions. They are related to the exotic multiplication (\ref{exoticm}) in the next section. Note that they would get involved in non-trivial extension of Mackey functors if we want to generalize to $C_{2^n},n\geq 3$. See \cite[Example 5.4]{basu21} for a reference. The author of that paper discovered this in the odd primary $p$ case, but we can consider $RO(C_{p^n})\text{'}\subset\text{'} RO(C_{2^n})$ as the index $2$ subgroup of orientable representations as in \cite{Zeng17}.

\begin{Prop}[The exotic extentions]
We have the following extensions
\begin{equation}
\begin{split}
    &0\to \Z\langle 4(\frac{\uta}{\ul})^k\rangle_{k\geq 2}[\uta]\xrightarrow{}\Z\langle 4(\frac{\uta}{\ul})^k\rangle_{k\geq 2}[\uta]\oplus\Z/2\langle (\frac{\atal}{\al})^k\rangle_{k\geq 2}[\uta]\\
    &\xrightarrow{} \Z/2\langle (\frac{\atal}{\al})^k\rangle_{k\geq 2}[\uta]\to 0.\label{exotic}
\end{split}
\end{equation}
\end{Prop}

\begin{proof}
This follows from (\ref{uext3}): $\frac{\atal}{\al}\in \HPt_V$ maps to $2\cdot\frac{\uta}{\ul}\in H^t_V$ means $(\frac{\atal}{\al})^2$ maps to $0$ as $H^t_V$ is $4$-torsion. This tells us the middle groups in (\ref{exotic}) are $\Z\oplus \Z/2$'s by the exact sequences $0\to\pi_V F=\Z/2\to \pi_VH\to \pi_VH^h=\Z\to 0$. This appeared in the exotic multiplication mentioned in \cite[Prop 6.9]{Zeng17}.

Another way to tell the exotic extension (\ref{exotic}) is that, suppose we have the extension
\begin{equation*}
\begin{split}
    0\to \Z\langle 4(\frac{\uta}{\ul})^k\rangle_{k\geq 2}\xrightarrow{2}\Z\langle 2(\frac{\uta}{\ul})^k\rangle_{k\geq 2}
    \xrightarrow{} \Z/2\langle (\frac{\atal}{\al})^k\rangle_{k\geq 2}\to 0.
\end{split}
\end{equation*}
instead. This means $2\cdot(\frac{\uta}{\ul})^k=(\frac{\atal}{\al})^k,k\geq 2$ in $H^t_{\bigstar}$, which is not true.
\end{proof}

For the rest of the terms, there is no extension for degree reasons. As a result, we have
\begin{Thm}
The top level of the $RO(C_4)$-graded Green functor $\underline{\pi_{\bigstar}}H\uZ$ is
\begin{align}
\begin{aligned}
\pi_{\bigstar}^{C_4/C_4}H\uZ&=\Z[\uta,\ul,\aal,\al]/(2\aal,4\al,\atal\ul-2\al\uta)\\
&\oplus\Z\langle 2\uta^{-i}\rangle_{i\geq 1}\\
&\oplus\Z\langle \uta^{-k}\ul^i\rangle_{i,k\geq 1}\\
&\oplus\Z\langle 2\uta^k\ul^{-1}\rangle_{k\geq 1}\\
&\oplus\Z\langle 4\ul^{-j}\rangle_{j\geq 1}[\uta^{-1}]\\
&\oplus\Z\langle 4\uta^{k}\ul^{-j}\rangle_{k\geq 1,j\geq 2}\\
&\oplus \Z/2\langle \Si\frac{\aal}{\al^i\ul^j}\rangle_{i,j\geq 1}[\uta^{\pm}]\\
&\oplus \Z/4\langle \Si\frac{1}{\al^i\ul^j}\rangle_{i\geq 1,j\geq 2}[\uta^{\pm}]\\
&\oplus \Z/4\langle \Si\frac{1}{\al^i\ul}\rangle_{i\geq 1}[\uta^{-1}]\\
&\oplus \Z/2\langle \Si\frac{\uta^j}{\al^i\ul}\rangle_{i,j\geq 1}\\
&\oplus \Z/2\langle \Si\frac{\aal}{\al^i\uta^j}\rangle_{i,j\geq 1}\\
&\oplus \Z/2\langle \Si\frac{1}{\al^i\uta^j}\rangle_{i,j\geq 1}\\
%&\oplus\Z/2\langle \frac{\aal^j}{\al}\rangle_{j\geq 3}[\uta]\\
&\oplus\Z/2\langle \aal^i\al^{-j}\rangle_{i\geq 3,j\geq 1}[\uta]\\
&\oplus\Z/2\Tre[\al^{\pm}]\\
&\oplus \Z/2\langle \frac{\aal\ul^i}{\uta^j}\rangle_{i,j\geq 1}[\al]\\
&\oplus\Z/2\langle 2\uta^{-i}\rangle_{i\geq 1}\langle \al^j\rangle_{j\geq 1}\\
&\oplus \Z/4\langle \uta^{-k}\ul^{i}\rangle_{i,k\geq 1}\langle \al^j\rangle_{j\geq 1}.\label{answer}
\end{aligned}
\end{align}
\end{Thm}

\begin{Rem}
\begin{enumerate}[label=(\alph*)]
\item Each $RO(G)$-degree comprises one of the $\Z,\Z/2,\Z/4$'s listed above, except the exotic ones listed in (\ref{exotic}), which are $\Z\oplus\Z/2$'s. This is very helpful when we determine the multiplicative structure.
\item For the torsion free elements, the number before them indicates their image in $H^h_{\bigstar}$. For example, $2\uta\ul^{-1}$ means the element that will map to $2\cdot \uta\ul^{-1}$ in $H^h_{\bigstar}$. It is also the element that satisfies $2\uta\ul^{-1}\cdot\ul=2\cdot\uta$.
\item There are some terms that can be collected together if the reader perfer. For example, the terms $\Z\langle \uta^{-k}\ul^{i}\rangle_{i,k\geq 1}$ and $\Z/4\langle \uta^{-k}\ul^{i}\rangle_{i,k\geq 1}\langle \al^j\rangle_{j\geq 1}$. But we choose to keep it to separate the torsion-free part and torsion part.
\end{enumerate}\label{Rem}
\end{Rem}

\section{Ring structure of \texorpdfstring{$\pi^{G/G}_{\bigstar}H\uZ$}{piH}}\label{ringstructure}
For the multiplications in (\ref{answer}), the author has determined the entire multiplicative structure, but the results are too messy to be presented here. We will just write down the general guidelines and show some specific examples. The reader should be able to easily determine the products of two arbitrary classes following this line.

Firstly, we have the positive cone which is a subring:
\[
\pi_{pos}H\uZ:=\underset{V \text{actual rep},n\in\Z}{\oplus}\pi_{n-V}^{G/G}H\uZ=\Z[\uta,\ul,\aal,\al]/(2\aal,4\al,\atal\ul=2\al\uta)
\]
and $\pi^{G/G}_{\bigstar}H\uZ$ is an algebra over it. The module structure over $\pi_{pos}H$ is straightforward as every spectrum in the Tate square (\ref{Tate}) is a $H\uZ$-module and every map including the connecting homomorphism is a $H\uZ$-module map. Thus we know multiplications by elements from the positive cone are by names, which are all we need in most applications. 

For products of general classes, except the ones in the exotic multiplications (\ref{exoticm}), Remark \ref{Rem} tells us that the candidates live in cyclic groups. Thus we can easily determine the relations by examining both sides.

For example, consider $\frac{\ul}{\uta}\cdot \Si\frac{1}{\al\ul^2}$, the only possible products are $\epsilon\cdot \Si\frac{1}{\al\ul}\uta^{-1},\epsilon =0,1$ for degree reason. But $\epsilon=0$ is impossible since we have $\frac{\ul}{\uta}\cdot \Si\frac{1}{\al\ul^2}\cdot\uta=\Si\frac{1}{\al\ul}\neq 0$. Another example is $\frac{\aal^3}{\al}\cdot \Si\frac{1}{\al^2\ul^2}=\epsilon\cdot\Si\frac{\aal\uta}{\al^2\ul^3},\epsilon=0,1$. Multiplying both sides by $\al$ rules out the possibility of $\epsilon=1$, since the left hand side will become $0$ (multiplication by $\aal^2$ will introduce $2$ by performing the gold relation, and $2$ kills $\alpha$) while the right hand side will not.

We also need to warn the reader that sometimes a more careful look is needed to determine the product.
For example, 
\[
\frac{2\uta}{\ul}\cdot\aal=0.
\]
Because of degree, $\frac{2\uta}{\ul}\cdot\aal$ can only be $\epsilon\cdot\frac{\aal^3}{\al},\epsilon=0,1$. If $\frac{2\uta}{\ul}\cdot\aal=\frac{\aal^3}{\al}$ is true, map it through $H\to \HPt\to H^t$ will give us $\frac{\aal^3}{\al}$, map it through $H\to H^h\to H^t$ will give us $0$ since in $H^h_{\bigstar}$, $\frac{2\uta}{\ul}\cdot\aal=2\cdot \frac{\uta}{\ul}\cdot\aal=0$. Thus we get a contradiction.

Products of classes involving only powers of $\ul$ and $\uta$ are determined by mapping them to $H^h_{\bigstar}$. For example $4\frac{\uta^2}{\uta^3}\cdot 4\frac{\uta^4}{\uta^1}=4\cdot 4\frac{\uta^6}{\uta^4}$, $2\uta^{-1}\cdot \frac{2\uta^2}{\ul}=2\cdot \frac{2\uta}{\ul}$.

\begin{Prop}
Except the exotic multiplications (\ref{exoticm}), products of the classes $\theta \uta^i\ul^j,\theta=1,2,4$ are determined by their images in $H^h_{\bigstar}$. Products of classes with $\Si$ in front of them are $0$.
\end{Prop}

\subsection{Multiplications involving \texorpdfstring{$\tre$}{tre}}
Here we have a complete list of multiplications involving the special class $\tre$. This class does not exist in the odd primary case of \cite{Zeng17}.

\begin{Prop}
All multiplications involving $\tre$ are:
\begin{enumerate}
    \item The module structure over $\pi_{pos}H$:

    \begin{enumerate}
        \item The ones that are determined by multiplication by names:
        \[
        \tre\frac{\al^k}{\aal^i\uta^j}\cdot\aal^r\cdot\uta^s\cdot\al^t=\tre\frac{\al^{k+t}}{\aal^{i-r}\uta^{j-s}},i-r,j-s\geq 0,k\in\Z,r,s,t\geq 0.
        \]
        \item The trivial products:
        \begin{align*}
        &\Tre[\al^{\pm}]\cdot\ul=0\\
        &\langle\tre\rangle[\aal^{-1},\al^{\pm}]\cdot\uta=0\\
        &\langle\tre\rangle[\uta^{-1}]\cdot\aal=0.
        \end{align*}
    
        \item The non-trivial ones are the implications of the following relation
        \begin{equation*}
        \tre\frac{1}{\uta^i\al^j}\cdot\aal=\Si\frac{1}{\uta^{i+1}\al^j},i\geq 0,j\geq 1. 
        \end{equation*}
        By implications we mean we can multiply or divide this relation by powers of $\aal,\al,\uta,\ul$ to get other relations.
    \end{enumerate}
    
    \item Multiplication by classes not coming from the positive cone:
        The only non-trivial multiplications are the implications of the following relation:
        \begin{equation*}
        \tre\cdot\frac{2\uta}{\ul}=\Si\frac{\aal}{\al\uta},
        \end{equation*}
        or equivalently, the relation
        \[
        \tre\frac{1}{\atal}\cdot\frac{2\uta}{\ul}=\tre\frac{1}{\al}.
        \]
\end{enumerate}    
\label{trelist}
\end{Prop}
\begin{proof}
All of the proofs are by careful analyses following the general guideline. To see the equivalence of the two relations in $(2)$, we only need to know $\tre\frac{1}{\al}\cdot\atal=\Si\frac{\aal}{\al\uta}$. This is true since the left hand side restricts to $0$ in $G/C_2$-level because of $\atal$, thus it is hit by a non-trivial multiplication by $\aal$.\\
\end{proof}

\begin{Rem}
Note here we have the multiplication:
\[
\tre\frac{1}{\al}\cdot \atal=\tre\cdot\frac{2\uta}{\ul}=\Si\frac{\aal}{\al\uta}.
\]
For elements from the exotic extensions (\ref{exotic}), we have
\[
\tre\frac{1}{\al^2}\cdot \atal^2=\tre\cdot\frac{4\uta^2}{\ul^2}=0.
\]
\end{Rem}

\subsection{The exotic multiplications}
\indent There are two cases where classes with different names can occupy the same $RO(G)$-degree. The first case is in the positive cone. For example $\atal\ul$ and $\al\uta$ live in the same degree, as well as multiples of these classes. But in these cases, there are still one generator: multiples of $\al\uta$.

The second case is the two groups of elements 
\[
\Z\langle 4(\frac{\uta}{\ul})^k\rangle_{k\geq 2}[\uta]\quad\text{and}\quad \Z/2\langle (\frac{\atal}{\al})^k\rangle_{k\geq 2}[\uta].
\] 
In these degrees, we have the exotic extensions (\ref{exotic}):
\begin{equation*}
\begin{split}
    &0\to \Z\langle 4(\frac{\uta}{\ul})^k\rangle_{k\geq 2}[\uta]\xrightarrow{}\Z\langle 4(\frac{\uta}{\ul})^k\rangle_{k\geq 2}[\uta]\oplus\Z/2\langle (\frac{\atal}{\al})^k\rangle_{k\geq 2}[\uta]\\
    &\xrightarrow{} \Z/2\langle (\frac{\atal}{\al})^k\rangle_{k\geq 2}[\uta]\to 0.
\end{split}
\end{equation*}
Let us consider $\Z\langle 2\frac{\uta}{\ul}\rangle$. $(2\frac{\uta}{\ul})^2$ lives in one of the above degrees, thus we can suppose
\[
(2\frac{\uta}{\ul})^2=a\cdot 4(\frac{\uta}{\ul})^2+b\cdot (\frac{\atal}{\al})^2,a\in\Z, b\in \Z/2.
\]
We know $4(\frac{\uta}{\ul})^2$ maps to $0$ through $H\to \HPt$ since it comes from $H_h$. $(\frac{\atal}{\al})^2$ maps to $0$ through $H\to H^h$ since it is $2$-torsion and in the coresponding degree in $H^h_{\bigstar}$, we have $\Z\langle (\frac{\uta}{\ul})^2\rangle$. Thus we deduce $a=1$ by mapping to $H^h$, and $b=1$ by mapping through $H\to \HPt\to H^t$. (Remember that as in (\ref{uext3}), $2\frac{\uta}{\ul}$ maps to $\frac{\atal}{\al}$ in $H^t_{\bigstar}$.)

Following the same argument, we can consider higher powers. Together with $\uta$-linearity, we have
\begin{Prop}
We have the following multiplication for $k\geq 2, i\geq 0$.
\begin{equation}
(2\frac{\uta}{\ul})^k\uta^i= 4(\frac{\uta}{\ul})^k\uta^i+(\frac{\atal}{\al})^k\uta^i.
\label{exoticm}
\end{equation}
\end{Prop}
This is mentioned in \cite[Prop 6.9]{Zeng17} in $C_{p^2}$, $p$ odd case. It is considered exotic because the product is a sum of two generators, while in other cases the product of any two classes is expressed in terms of a single generator. The exotic ones are the only situation where this will happen. This is simply because the class $\frac{2\uta}{\ul}$ has a non-trivial image $\frac{\atal}{\al}$ in $H^t_{\bigstar}$, while other possible torsion free elements in these degrees $\frac{4\uta^k}{\ul^j},k\geq j\geq 2$ has a trivial image in $H^t_{\bigstar}$.\\

For example, we have
\[
\frac{2\uta}{\ul}\cdot 4(\frac{\uta}{\ul})^2=2\cdot 4(\frac{\uta}{\ul})^3
\]
because $4(\frac{\uta}{\ul})^2$ maps to $0$ through $H\to \HPt$ (its image through $H_{\bigstar}\to H^h_{\bigstar}\to H^t_{\bigstar}$ is 0 as $H^t_{\bigstar}$ is $4$-torsion). Similarly, another example is $
\frac{4\uta^3}{\ul^3}\cdot \frac{\ul}{\uta}=4\frac{\uta^2}{\ul^2}$.

\subsection{Reading off the localizations \texorpdfstring{$H[\aal^{-1}]_{\bigstar}$}{aal} and \texorpdfstring{$H[\al^{-1}]_{\bigstar}$}{al}} Our computation is based on localizations at the Euler classes $\al,\aal$. One we know the ring structure, it is very easy to go the other way to read off the localized rings $H[\aal^{-1}]_{\bigstar}$ and $H[\al^{-1}]_{\bigstar}$. 

Firstly, for the geometric fixed point $\HPf=H[\aal^{-1}]$. The only terms in $\pi_{\bigstar}H$ that are not killed by powers of $\aal$ are the positive cone and the terms
\[
\frac{\aal^i}{\al^j}[\uta],i\geq 3,j\geq 1.
\]
For all of the rest of the terms, we have to perform the gold relation once we get a term involving $\aal^2$. Then we get the factor $2$. Thus they will be killed if multiplied by $\aal$ again. (Most of them are killed by $\aal^3$, terms involving $\tre$ may be killed by $\aal^4$ by $(1)(c)$ in (\ref{trelist}).)

As a result, we have
\begin{equation*}
\pi_{\bigstar}\HPf=\Z/2[\uta,\atal^{\pm},\al^{\pm}].\
\end{equation*}
as in (\ref{geometric}).

Secondly, for $\HPt$, we can see that the terms not killed by a power of $\al$ in (\ref{answer}) are: the positive cone, $\Z\langle 2\uta^{-i}\rangle_{i\geq 1},\Z\langle \uta^{-k}\ul^i\rangle_{i,k\geq 1},\Z\langle 2\uta^k\ul^{-1}\rangle_{k\geq 1}$, and all terms in the last five summands of (\ref{answer}). Note that we have $\frac{2\uta}{\ul}\cdot\al=\atal$ just by multiplying both sides by $\ul$. As a result, we get $\pi_{\bigstar}\HPt$.

\section{Mackey functor structure}
In this section we determine the $RO(C_4)$-graded Green functor structure of $\underline{\pi_{\bigstar}}H\uZ$. For the definitions of $RO(G)$-graded Mackey functors and Green functors, see \cite[Sec 2,3]{LM06}.

\subsection{Generalities 
on Mackey functors and recollections of the \texorpdfstring{$C_2$}{C2}-computation} Firstly, we recall the abstract properties of a cohomological Green functor $\underline{R}$ (\cite{Web00,Zeng17}) that we will use:
\begin{enumerate}
    \item Restrictions and Weyl actions are ring homomorphisms.
    \item (cohomological condition) 
    \[
    tr^H_K\circ res^H_K=|H/K|,K\subset H\subset G.
    \]
    \item (double coset formula) 
    \[
    res^H_J\circ tr^H_K=\underset{x\in[J\backslash H/K]}{\Sigma}tr^J_{J\cap ^xK}\circ c_x\circ res^K_{J^x\cap K},J,K\subset H.
    \]
    \item (Frobenius reciprocity)
    \[
    tr^H_K(a\cdot res^H_K(b))=tr^H_K(a)\cdot b, K\subset H,a\in R(G/K),b\in R(G/H).
    \]
\end{enumerate}

In addition to the above, for any $C_4$-Mackey functor $\underline{M}$, we have
\begin{equation}
    \begin{aligned}
    &\text{im}(\res)\subset \underline{M}(C_4/C_2)^{C_2'},\\ &\text{im}(res^4_1)\subset \underline{M}(C_4/e)^{C_4}\label{fixed}.
    \end{aligned}
\end{equation}
This is deduced from applying $\underline{M}$ to the following commutative diagram:
\[
\xymatrix{
G/G\ar[r]^{c_{\gamma}=id}& G/G\\
G/H\ar[r]^{c_\gamma}\ar[u]& G/H\ar[u]
}
\]
for subgroups $H=e, C_2$. $c_{\gamma}(xH)=\gamma^3xH, x\in G$. 

For the Mackey functor structures related to the orientation classes $\uta,\ul$ and Euler classes $\aal,\al$(\cite[Prop 3.3, Def 3.4]{HHRb}), we have:
\begin{equation}
\xymatrix{
	\Z\langle\uta\rangle\ar@/_/[d]_{1}&\Z\langle\ul\rangle\ar@/_/[d]_{1}&\Z/2\langle\aal\rangle\ar@/_/[d]_{0}&\Z/4\langle\al\rangle\ar@/_/[d]_{1}\\
	\Z\langle\widebar{\uta}\rangle\ar@/_/[u]_{2}\ar@/_/[d]_{1}&\Z\langle\widebar{\ul}\rangle\ar@/_/[u]_{2}\ar@/_/[d]_{1}&0\ar@/_/[u]_{0}\ar@/_/[d]_{0}&\Z/2\langle\ats\rangle\ar@/_/[d]_{0}\ar@/_/[u]_{2}.\\
	\Z\langle\widebar{\widebar{\uta}}\rangle\ar@/_/[u]_{2}&\Z\langle\widebar{\widebar{\ul}}\rangle\ar@/_/[u]_{2}&0\ar@/_/[u]_{0}&0\ar@/_/[u]_{0}
}
\end{equation}
Note in the above, three of the restrictions are $1$ by the definition of these classes, and the transfers are deduced from the cohomological condition.

Now we keep the picture in mind:
\begin{equation}
    \underline{\pi}_{*+*\alpha+*\lambda}H\uZ=
\xymatrix{
\pi^{G/G}_{*+*\alpha+*\lambda}H\ar@/_/[d]_{\res}\\
\pi^{C_2/C_2}_{*+2*\sigma}H\ar@/_/[u]_{\tr}[\eal^{\pm}]\ar@/_/[d]_{res^2_1}\acts\gamma\\
\pi^{e/e}_{*}H\ar@/_/[u]_{tr^2_1}[\overline{\eal}^{\pm},e_{\lambda}^{\pm}]\acts\gamma.
}\label{mack}
\end{equation}
We would like to consider this picture as three commutative rings connected by restrictions and transfers. 

We quote the $C_2$-equivariant computation from \cite[Prop 6.5]{Zeng17}:
\begin{Prop}
\begin{equation}
    \begin{aligned}
    \pi_{*+*\sigma}^{C_2/C_2}H\uZ&=\Z[\uts,a_{\sigma}]/(2a_{\sigma})\\
    &\oplus \Z\langle 2\uts^{-i}\rangle_{i\geq 1}\\
    &\oplus \Z/2\langle \Si\uts^{i}a_{\sigma}^{-j}\rangle_{i,j\geq 1}
    \end{aligned}\label{C2}
\end{equation}
\end{Prop}
Note here we only care about $RO(C_4)$-grading, and the image of $i^*_{C_2}:RO(C_4)\to RO(C_2)$ is the subring $\{m+2n\sigma|m,n\in\Z\}\subset RO(C_2)$. Thus the only $C_2$-Mackey functors related to us are:
\begin{equation*}
\xymatrix{
	\Z[\uts]\ar@/_/[d]_{1}&\Z/2[\uts]\langle \ats^i\rangle_{i\geq 1}\ar@/_/[d]_{0}&\Z\langle 2\uts^{-i}\rangle_{i\geq 1}\ar@/_/[d]_{2}&\Z/2\langle \Si\uts^{-i}\ats^{-j}\rangle_{i,j\geq 1}\ar@/_/[d]_{0}\\
	\Z[\widebar{\uts}]\ar@/_/[u]_{2}&0\ar@/_/[u]_{0}&\Z\langle {u_{\sigma}}^{-2i}\rangle_{i\geq 1}\ar@/_/[u]_{1}&0\ar@/_/[u]_{0}\\
}
\end{equation*}
A reminder for the reader is that in the $C_2$-equivariant computation, there are the classes ${u_{\sigma}}^{\pm i}$ with $i$ odd that are not covered by non-trivial classes or torsion free classes on the $C_2/C_2$-level. The two cases are
\begin{equation*}
\xymatrix{
	0\ar@/_/[d]_{0}&\Z/2\langle \Si\uts^{-i}a_{\sigma}^{-1}\rangle_{i\geq 1}\ar@/_/[d]_{0}\\
	\Z\langle {u_{\sigma}}^{-1}\rangle[\uts]\ar@/_/[u]_{0}&\Z\langle {u_{\sigma}}^{-3}\rangle[\uts^{-1}]\ar@/_/[u]_{1}.\\
}
\end{equation*}
That the transfers in the second case are $1$'s can be seen by noticing that $\Z/2\langle \Si\uts^{-i}a_{\sigma}^{-1}\rangle$ are killed by $a_{\sigma}$. The Weyl group actions are determined by: the generator of $C_2$ act by $-1$ on ${u_{\sigma}}$ (and thus $1$ on $\widebar{\uts}$).

Here in the $C_4$-equivariant computation, we have the analogous cases:
\[
\xymatrix{
	0\ar@/_/[d]_{0}&\Z/2\langle \tre\rangle[\uta^{-1}]\ar@/_/[d]_{0}\\
	\Z\langle \eal\rangle[\etal^{-1}]\ar@/_/[u]_{0}\ar@/_/[d]_{1}&\Z\langle \ethal\rangle\ar@/_/[u]_{1}[\etal]\ar@/_/[d]_{1}\\
	\Z\langle \widebar{\eal}\rangle[\widebar{\etal}^{-1}]\ar@/_/[u]_{2}&\Z\langle \widebar{\ethal}\rangle[\widebar{\etal}]\ar@/_/[u]_{2}.
}
\]
When possible, we choose to use $\etal$ rather than $\widebar{\uta}^{-1}$ to avoid having too much bars. Of course, $\ual^{-2}$ is another way to do so.

The above $C_2$-computation tells us about $res^2_1,tr^2_1$, i.e., the two lower levels in our $C_4$-Mackey functor computation. It only concerns with $\{*+*\lambda\}$-grading. We extend the result to $RO(C_4)$-grading by $\eal$-periodicity of the two lower levels. Restrictions commute with the products, thus knowing $\widebar{\eal}=res^2_1(\eal)$ is enough. The problem is to extend $tr^2_1$ to $RO(C_4)$-grading. There are two ways to do this:
\begin{enumerate}
    \item By Frobenius reciprocity, we have $tr^2_1(\widebar{\eal}\cdot x)=\eal\cdot tr^2_1(x),x\in \pi^{G/e}_VH$. Thus we can treat $tr^2_1$ as linear over $\eal$.
    \item Using the same argument as Lemma \ref{lemrestr} below with the cofiber sequence
    \[
    {C_2/e}_+\xrightarrow{res^2_1} S^0\xrightarrow{a_{\sigma}} S^{\sigma},
    \]
    we can show $im(tr^2_1)=ker(a_{\sigma})$. Thus a non-trivial $tr^2_1$ is the same as an trivial $a_{\sigma}$ multiplication, which is the same as a trivial $a_{\sigma}\eal^i$-multiplication, $i\in \Z$. And we can extend $tr^2_1$ by $\eal$-periodicity.
\end{enumerate}

When it comes to $\res,\tr$, we use the following cofiber sequences:
\begin{equation*}
	\begin{split}
	&S^{-1}\xrightarrow{\aal}S^{\alpha-1}\to G/{C_2}_+\xrightarrow{res}S^0\xrightarrow{\aal}S^{\alpha},\\
	&S^1\xleftarrow{\aal}S^{1-\alpha}\xleftarrow{}G/{C_2}_+\xleftarrow{tr}S^0\xleftarrow{\aal}S^{-\alpha}.
	\end{split}
\end{equation*}
Applying $[-,S^{-V}\wedge H\uZ]^G$ give us exact sequences
\begin{equation}
	\begin{split}
	&\pi_{V-1}H\xleftarrow{\aal}\pi_{V+\alpha-1}H\xleftarrow{}\pi^{G/{C_2}}_VH\xleftarrow{\res}\pi_V H\xleftarrow{\aal}\pi_{V+\alpha}H,\\
	&\pi_{V+1}H\xrightarrow{\aal}\pi_{V+1-\alpha}\to \pi_V^{G/{C_2}}H\xrightarrow{\tr}\pi_VH\xrightarrow{\aal}\pi_{V-\alpha}H.
	\end{split}
\end{equation}
which tells us
\begin{Lem}
$ker(\aal)=im(\tr)$ and $im(\aal)=ker(\res)$.\label{lemrestr}
\end{Lem}
This idea comes from \cite[Lem 4.2]{HHRb}. Following their method, we can also derive this result just from the first cofiber sequence by noticing a shearing isomorphism:
\begin{equation}
	\begin{split}
	\xymatrix{
	\pi_{V-1}H&\pi_{V+\alpha-1}H\ar[l]_-{\aal}&\pi^{G/{C_2}}_VH\ar[l]\ar[d]^{\cong}_{\eal}&\pi_V H\ar[l]_-{\res}&\pi_{V+\alpha}H\ar[l]_-{\aal},\\
	&&\pi^{G/{C_2}}_{V+\alpha-1}H\ar[ul]^{\tr}
	}
	\end{split}
\end{equation}

\subsection{Complete \texorpdfstring{$C_4$}{C4}-Mackey functor structure}
Now we list all the Mackey functors involved in $\underline{\pi^{G/G}_{\bigstar}}H\uZ$ in the order of (\ref{answer}). The non-trivial Mackey functors that has trivial $G/G$-level will be listed after this. We will number these Mackey functors and explain the methods to determine them below.

First note that the positive cone additively splits as
\begin{equation*}
    \begin{aligned}
    \pi_{pos}H\uZ&=\Z[\uta,\ul,\aal,\al]/(2\aal,4\al,\atal\ul=2\al\uta)\\
    &=\Z[\uta,\ul]\\
    &\oplus \Z/4[\uta,\ul]\langle \al^i\rangle_{i\geq 1}\\
    &\oplus \Z/2[\uta,\al]\langle \aal^i\rangle_{i\geq 1}\\
    &\oplus \Z/2[\uta,\al]\langle \ul^i\rangle_{i\geq 1}\langle \aal\rangle.
    \end{aligned}
\end{equation*}

\begin{Thm}
The Mackey functors involved in $\underline\pi_{\bigstar}H\uZ$ are the ones listed below, from $(1)$ to $(38)$.
\end{Thm}

The Mackey functors with non-trivial $G/G$-levels are:

\begin{equation*}
\xymatrix{
\Z[\uta,\ul]\ar@/_/[d]_{1}&\Z/4[\uta,\ul]\langle \al^i\rangle_{i\geq 1}\ar@/_/[d]_{1}&\Z/2[\uta,\al]\langle \aal^i\rangle_{i\geq 1}\ar@/_/[d]_{0}&\Z/2[\uta,\al]\langle \ul^i\rangle_{i\geq 1}\langle \aal\rangle\ar@/_/[d]_{0}\\
\Z[\widebar{\uta},\uts]\ar@/_/[u]_{2}\ar@/_/[d]_{1}&\Z/2[\widebar{\uta},\uts]\langle \ats^i\rangle_{i\geq 1}\ar@/_/[u]_{2}\ar@/_/[d]_{0}&0\ar@/_/[u]_{0}\ar@/_/[d]_{0}&0\ar@/_/[u]_{0}\ar@/_/[d]_{0}\\
\Z[\widebar{\widebar{\uta}},\widebar{\uts}]\ar@/_/[u]_{2}&0\ar@/_/[u]_{0}&0\ar@/_/[u]_{0}&0\ar@/_/[u]_{0}\\
(1) &(2) &(3) &(4)
}
\end{equation*}

\begin{equation*}
\xymatrix{
\Z\langle 2\uta^{-i}\rangle_{i\geq 1}\ar@/_/[d]_{2}&\Z\langle \uta^{-k}\ul^i\rangle_{i,k\geq 1}\ar@/_/[d]_{1}&\Z\langle 2\uta^k\ul^{-1}\rangle_{k\geq 1}\ar@/_/[d]_{1}\\
\Z\langle \etal^i\rangle_{i\geq 1}\ar@/_/[u]_{1}\ar@/_/[d]_{1}&\Z\langle \etal^k\cdot\uts^i\rangle_{i,k\geq 1}\ar@/_/[u]_{2}\ar@/_/[d]_{1}&\Z\langle 2\uts^{-1}\cdot\widebar{\uta}^k\rangle_{k\geq 1}\ar@/_/[u]_{2}\ar@/_/[d]_{2}\\
\Z\langle \widebar{\etal}^i\rangle_{i\geq 1}\ar@/_/[u]_{2}&\Z\langle \widebar{\etal}^k\cdot\uts^i\rangle_{i,k\geq 1}\ar@/_/[u]_{2}&\Z\langle \widebar{\uts}^{-1}\cdot\widebar{\widebar{\uta}}\,^k\rangle_{k\geq 1}\ar@/_/[u]_{1}.\\
(5) &(6) &(7)
}
\end{equation*}

\begin{equation*}
\xymatrix{
\Z\langle 4\ul^{-j}\rangle_{j\geq 1}[\uta^{-1}]\ar@/_/[d]_{2}&\Z\langle 4\uta^k\ul^{-j}\rangle_{k\geq 1,j\geq 2}\ar@/_/[d]_{2}&\Z/2\langle \Si\frac{\aal}{\al^i\ul^j}\rangle_{i\geq 2,j\geq 1}[\uta^{\pm}]\ar@/_/[d]_{0}\\
\Z\langle 2\uts^{-j}\rangle_{j\geq 1}[\etal]\ar@/_/[u]_{1}\ar@/_/[d]_{2}&\Z\langle 2\uts^{-j}\cdot\widebar{\uta}^k\rangle_{k\geq 1,j\geq 2}\ar@/_/[u]_{1}\ar@/_/[d]_{2}&0\ar@/_/[u]_{0}\ar@/_/[d]_{0}\\
\Z\langle \widebar{\uts}^{-j}\rangle_{j\geq 1}[\widebar{\etal}]\ar@/_/[u]_{1}&\Z\langle \widebar{\uts}^{-j}\cdot\widebar{\widebar{\uta}}\,^k\rangle_{k\geq 1,j\geq 2}\ar@/_/[u]_{1}&0\ar@/_/[u]_{0}\\
(8) &(9) &(10)
}
\end{equation*}

\begin{equation*}
\xymatrix{
\Z/2\langle \Si\frac{\aal}{\al\ul^j}\rangle_{j\geq 1}[\uta^{\pm}]\ar@/_/[d]_{0}&\Z/4\langle \Si\frac{1}{\al^i\ul^j}\rangle_{i\geq 1,j\geq 2}[\uta^{\pm}]\ar@/_/[d]_{1}&\Z/4\langle \Si\frac{1}{\al^i\ul}\rangle_{i\geq 1}[\uta^{-1}]\ar@/_/[d]_{1}\\
\Z\langle 2\uts^{-2}\eal \rangle[\widebar{\uta}^{\pm},\uts^{-1}]\ar@/_/[u]_{1}\ar@/_/[d]_{2}&\Z/2\langle \Si\frac{1}{\ats^i\uts^j}\rangle_{i\geq 1,j\geq 2}[\etal^{\mp}]\ar@/_/[u]_{2}\ar@/_/[d]_{0}&\Z/2\langle \Si\frac{1}{\ats^i\uts}\rangle_{i\geq 1}[\etal]\ar@/_/[u]_{2}\ar@/_/[d]_{0}\\
\Z\langle \widebar{\uts}^{-2}\widebar{\eal} \rangle[\widebar{\widebar{\uta}}\,^{\pm},\widebar{\uts}^{-1}]\ar@/_/[u]_{1}&0\ar@/_/[u]_{0}&0\ar@/_/[u]_{0}\\
(11) &(12) &(13)
}
\end{equation*}

\begin{equation*}
\xymatrix{
\Z/2\langle \Si\frac{\uta^j}{\al^i\ul}\rangle_{i,j\geq 1}\ar@/_/[d]_{1}&\Z/2\langle \Si\frac{\aal}{\al\uta^j}\rangle_{j\geq 1}\ar@/_/[d]_{0}\\
\Z/2\langle \Si\frac{1}{\ats^i\uts}\cdot\widebar{\uta}^j\rangle_{i,j\geq 1}\ar@/_/[u]_{0}\ar@/_/[d]_{0}&\Z\langle 2\uts^{-1}\cdot\eal\rangle[\etal]\ar@/_/[u]_{1}\ar@/_/[d]_{2}\\
0\ar@/_/[u]_{0}&\Z\langle \widebar{\uts}^{-1}\cdot\widebar{\eal}\rangle[\widebar{\etal}]\ar@/_/[u]_{1}\\
(14) &(15)
}
\end{equation*}

\begin{equation*}
\xymatrix{
\Z/2\langle \Si\frac{\aal}{\al^i\uta^j}\rangle_{i\geq 2,j\geq 1}\ar@/_/[d]_{0}&\Z/2\langle \Si\frac{1}{\al^i\uta^j}\rangle_{i,j\geq 1}\ar@/_/[d]_{0}\\
0\ar@/_/[u]_{0}\ar@/_/[d]_{0}&0\ar@/_/[u]_{0}\ar@/_/[d]_{0}\\
0\ar@/_/[u]_{0}&0\ar@/_/[u]_{0}\\
(16) &(17)
}
\end{equation*}

\begin{equation*}
\xymatrix{
\Z/2\langle \frac{\aal^3}{\al}\rangle[\aal,\uta]\ar@/_/[d]_{0}&\Z/2\langle \frac{\aal^3}{\al^2}\rangle[\al^{-1}][\uta]\ar@/_/[d]_{1}&\Z/2\langle \frac{\aal^4}{\al^3}\rangle[\al^{-1}][\uta][\frac{\atal}{\al}]\ar@/_/[d]_{0}\\
0\ar@/_/[u]_{0}\ar@/_/[d]_{0}&\Z/2\langle \Si\frac{1}{\uts\ats}\cdot u_{3\alpha}\rangle[\ats^{-1}][\widebar{\uta}]\ar@/_/[u]_{0}\ar@/_/[d]_{0}&0\ar@/_/[u]_{0}\ar@/_/[d]_{0}\\
0\ar@/_/[u]_{0}&0\ar@/_/[u]_{0}&0\ar@/_/[u]_{0}\\
(18) &(19) &(20)
}
\end{equation*}

\begin{equation*}
\xymatrix{
\Z/2\langle \frac{\aal^5}{\al^3}\rangle[\al^{-1}][\uta][\frac{\atal}{\al}]\ar@/_/[d]_{0}&\Z/2\langle \frac{\aal^4}{\al^2}\rangle[\uta][\frac{\atal}{\al}]\ar@/_/[d]_{0}\\
\Z/2\langle \Si\frac{1}{\uts^2\ats}\cdot u_{5\alpha}\rangle[\ats^{-1}][\widebar{\uta}][\frac{\widebar{\uta}}{\uts}]\ar@/_/[u]_{0}\ar@/_/[d]_{0}&\Z\langle 2\uts^{-2}\cdot\widebar{\uta}\rangle[\widebar{\uta}][\frac{\widebar{\uta}}{\uts}]\ar@/_/[u]_{0}\ar@/_/[d]_{2}\\
0\ar@/_/[u]_{0}&\Z\langle \widebar{\uts}^{-2}\cdot\widebar{\widebar{\uta}}\rangle[\widebar{\widebar{\uta}}][\frac{\widebar{\widebar{\uta}}}{\widebar{\uts}}]\ar@/_/[u]_{1}\\
(21) &(22)
}
\end{equation*}

\begin{equation*}
\xymatrix{
\Z/2\langle \frac{\aal^5}{\al^2}\rangle[\uta][\frac{\atal}{\al}][\aal]\ar@/_/[d]_{0}&\Z/2\langle \tre\rangle[\uta^{-1}]\ar@/_/[d]_{0}\\
0\ar@/_/[u]_{0}\ar@/_/[d]_{0}&\Z\langle \ethal\rangle[\etal]\ar@/_/[u]_{1}\ar@/_/[d]_{1}\\
0\ar@/_/[u]_{0}&\Z\langle \widebar{\ethal}\rangle[\widebar{\etal}]\ar@/_/[u]_{2}\\
(23) &(24)
}
\end{equation*}

\begin{equation*}
\xymatrix{
\Z/2\langle \tre\rangle[\uta^{-1}]\langle \al^i\rangle_{i\geq 1}\ar@/_/[d]_{0}&\Z/2\langle \tre\rangle[\uta^{-1}]\langle \frac{1}{\al^k}\rangle_{k\geq 1}\ar@/_/[d]_{0}\\
\Z/2\langle \ethal\rangle[\etal]\langle \ats^i\rangle_{i\geq 1}\ar@/_/[u]_{1}\ar@/_/[d]_{0}&0\ar@/_/[u]_{0}\ar@/_/[d]_{0}\\
0\ar@/_/[u]_{0}&0\ar@/_/[u]_{0}\\
(25) &(26)
}
\end{equation*}

\begin{equation*}
\xymatrix{
\Z/2\langle \tre\rangle[\uta^{-1}]\langle (\frac{\al}{\atal})^i\rangle_{i\geq 1}\ar@/_/[d]_{0}&\Z/2\langle \tre\rangle[\uta^{-1}]\langle (\frac{\al}{\atal})^i\cdot\al^k\rangle_{i,k\geq 1}\ar@/_/[d]_{0}\\
\Z\langle \ethal\rangle[\etal]\langle (\uts\cdot\etal)^i\rangle_{i\geq 1}\ar@/_/[u]_{0}\ar@/_/[d]_{1}&\Z/2\langle \ethal\rangle[\etal]\langle (\uts\cdot\etal)^i\cdot\ats^k\rangle_{i,k\geq 1}\ar@/_/[u]_{0}\ar@/_/[d]_{0}\\
\Z\langle \widebar{\ethal}\rangle[\widebar{\etal}]\langle (\widebar{\uts}\cdot\widebar{\etal})^i\rangle_{i\geq 1}\ar@/_/[u]_{2}&0\ar@/_/[u]_{0}\\
(27) &(28)
}
\end{equation*}

\begin{equation*}
\xymatrix{
\Z/2\langle \tre\rangle[\uta^{-1}]\langle (\frac{\al}{\atal})^i\cdot\al^{-k}\rangle_{i,k\geq 1}\ar@/_/[d]_{0}&\Z/2\langle \frac{\aal\ul^i}{\uta^j}\rangle_{i,j\geq 1}[\al]\ar@/_/[d]_{0}\\
0\ar@/_/[u]_{0}\ar@/_/[d]_{0}&0\ar@/_/[u]_{0}\ar@/_/[d]_{0}\\
0\ar@/_/[u]_{0}&0\ar@/_/[u]_{0}\\
(29) &(30)
}
\end{equation*}

\begin{equation*}
\xymatrix{
\Z/2\langle 2\uta^{-i}\rangle_{i\geq 1}\langle \al^j\rangle_{j\geq 1}\ar@/_/[d]_{0}&\Z/4\langle \uta^{-k}\ul^i\rangle_{i,k\geq 1}\langle \al^j\rangle_{j\geq 1}\ar@/_/[d]_{1}\\
\Z/2\langle \etal^i\rangle_{i\geq 1}\langle \ats^j\rangle_{j\geq 1}\ar@/_/[u]_{1}\ar@/_/[d]_{0}&\Z/2\langle \etal^k\cdot\uts^i\rangle_{i,k\geq 1}\langle \ats^j\rangle_{j\geq 1}\ar@/_/[u]_{2}\ar@/_/[d]_{0}\\
0\ar@/_/[u]_{0}&0\ar@/_/[u]_{0}\\
(31) &(32)
}
\end{equation*}

\begin{Rem}
Here we want to mention that in the degrees involved in the exotic extensions (\ref{exotic}), the Mackey functors split into direct sums in the category of $C_4$-Mackey functors of $(8)$ when $k\geq j\geq 2$ and $(19)$ with $G/C_2$,$G/e$-level replaced by $0$:
{\scriptsize
\begin{equation*}
\xymatrix@C=0.3em{
\Z\langle 4(\frac{\uta}{\ul})^k\rangle_{k\geq 2}[\uta]\oplus\Z/2\langle (\frac{\atal}{\al})^k\rangle_{k\geq 2}[\uta]\ar@/_/[d]_{(2\quad 0)}\,=&\Z\langle 4(\frac{\uta}{\ul})^k\rangle_{k\geq 2}[\uta]\ar@/_/[d]_2&\Z/2\langle (\frac{\atal}{\al})^k\rangle_{k\geq 2}[\uta]\ar@/_/[d]_{0}\\
\Z\langle 2\uts^{-k}\cdot\widebar{\uta}^k\rangle[\etal^{-1}]\ar@/_/[u]_{(1\quad 0)}\ar@/_/[d]_{2}&\Z\langle 2\uts^{-k}\cdot\widebar{\uta}^k\rangle[\etal^{-1}]\ar@/_/[u]_{1}\ar@/_/[d]_{2}\quad\quad\quad\oplus&0\ar@/_/[d]_{0}\ar@/_/[u]_{0}\\
\Z\langle \widebar{\uts^{-k}}\cdot\widebar{\widebar{\uta}}\,^k\rangle[\widebar{\etal}^{-1}]\ar@/_/[u]_{1}&\Z\langle \widebar{\uts^{-k}}\cdot\widebar{\widebar{\uta}}\,^k\rangle[\widebar{\etal}^{-1}]\ar@/_/[u]_{1}&0\ar@/_/[u]_{0}\\
}
\end{equation*}
}
If we want to generalize to $C_{2^n},n\geq 3$, the Mackey functors in the corresponding degrees do not split, see \cite{basu21}.
\end{Rem}

By checking the $G/C_2$-level of all the Mackey functors listed above, we find that the non-trivial Mackey functors that have trivial $G/G$-levels are (the two $i$'s in $(33)$ take the same value, same for others):
\begin{equation*}
\xymatrix{
0\ar@/_/[d]_{0}&0\ar@/_/[d]_{0}\\
\Z\langle\uts^i \rangle_{i\geq 0}\langle e_{(1+2i)\alpha}\rangle_{i\geq 0}[\widebar{\uta}]\ar@/_/[u]_{0}\ar@/_/[d]_{1}&\Z/2\langle\uts^j\rangle_{j\geq 1}\langle \ats^i\rangle_{i\geq 1}\langle e_{(2j+1)\alpha}\rangle_{j\geq 0}[\widebar{\uta}]\ar@/_/[u]_{0}\ar@/_/[d]_{0}\\
\Z\langle\widebar{\uts}^i \rangle_{i\geq 0}\langle \widebar{e_{(1+2i)\alpha}}\rangle_{i\geq 0}[\widebar{\widebar{\uta}}]\ar@/_/[u]_{2}&0\ar@/_/[u]_{0}\\
(33) &(34)
}
\end{equation*}

\begin{equation*}
\xymatrix{
0\ar@/_/[d]_{0}&0\ar@/_/[d]_{0}\\
\Z/2\langle \ats^i\rangle_{i\geq 1}\langle \eal\rangle_{j\geq 0}[\widebar{\uta}]\ar@/_/[u]_{0}\ar@/_/[d]_{0}&\Z\langle 2\uts^{-1}\rangle\langle \ual\rangle[\widebar{\uta}]\ar@/_/[u]_{0}\ar@/_/[d]_{2}\\
0\ar@/_/[u]_{0}&\Z\langle \widebar{\uts}^{-1}\rangle\langle \widebar{\ual}\rangle[\widebar{\widebar{\uta}}]\ar@/_/[u]_{1}\\
(35) &(36)
}
\end{equation*}

\begin{equation*}
\xymatrix{
0\ar@/_/[d]_{0}&0\ar@/_/[d]_{0}\\
\Z/2\langle \Si\frac{1}{\ats^i\uts}\rangle_{i\geq 1}\langle \ual\rangle[\etal]\ar@/_/[u]_{0}\ar@/_/[d]_{2}&\Z/2\langle \Si\frac{1}{\ats^i\uts^j}\rangle_{i\geq 1,j\geq 2}\langle u_{(2j-1)\alpha}\rangle[\etal]\ar@/_/[u]_{0}\ar@/_/[d]_{2}\\
0\ar@/_/[u]_{0}&0\ar@/_/[u]_{1}\\
(37) &(38)
}
\end{equation*}

By a careful degree counting, there's no Mackey functor having trivial $C_4/C_4$ and $C_4/C_2$-levels and non-trivial $C_4/e$-level.

Now we show how to determine $\res,\tr$ and Weyl actions.
\begin{itemize}
    \item $(1)-(4)$ are determined by definition.
    \item $(5)-(9)$ are determined by using the same method. Let us take $4\langle \ul^{-j}\rangle[\uta^{-1}]$ in $(8)$-as an example. By $\uta,\ul$-divisibility (multiplication by any of these two classes is an isomorphism on $4\ul^{-i}\uta^{-j},i,j\geq 1$), we only need to determine the transfer and restriction for $4\ul^{-1}$. This class satisfies $4\ul^{-1}\cdot \ul=4$. Applying $\res$ to this identity and using the fact that $\res(\ul)=\uts$ tells us that $\res(4\ul^{-1})=2\cdot 2\uts^{-1}$. Another way to see this is to map to $H^h$, and we have the following diagram
    \[
    \xymatrix{
    \Z\langle 4\ul^{-1}\rangle=H_{\lambda-2}\ar[d]_{\res}\ar[r]^-4 &H^0(C_4;\pi^{G/e}_{\lambda-2})=\Z\langle \ul^{-1}\rangle\ar[d]_{1}\\
    \Z\langle 2\uts^{-1}\rangle=H_{2\sigma-2}\ar[r]^-2&H^0(C_2;\pi^{G/e}_{2\sigma-2})=\Z\langle \uts^{-1}\rangle.
    }
    \]
    $\tr$ follows from the cohomological condition.
    \item $(11)$. The top classes are killed by $\aal$.
    \item $(12),(13)$. By the cohomological condition $\tr\circ\res=2$, we know both $\tr$ and $\res$ are non-zero. Now we can deduce $\res=1$ from (\ref{lemrestr}) by noticing the top level classes are not divisible by $\aal$, or deduce $\tr=2$ by noticing $2$ times of the top level classes are killed by $\aal$. 
    
    Here to prove the top level classes are not divisible by $\aal$, we use our old trick of formal operation. Take $\Si\frac{1}{\al^i\ul^j}$ as an example. Formally (modulo the constant $2$ in the gold relation) we have $\Si\frac{1}{\al^i\ul^j\aal}=\Si\frac{\aal}{\al^i\ul^j\atal}=\Si\frac{\aal}{\al^{i+1}\ul^{j-1}\uta}$. So we wonder whether we have $\Si\frac{\aal}{\al^{i+1}\ul^{j-1}\uta}\cdot\aal=\Si\frac{1}{\al^i\ul^j}$, which is wrong because the left hand side is $2$ times of the right hand side.
    \item $(14)$. Top classes are not hit (divisible) by $\aal$.
    \item $(15)$. Top classes are killed by $\aal$.
    \item $(19)-(22)$. Top classes are not killed by $\aal$. In $(19)$, top classes are not hit by $\aal$, while in $(21),(22)$, top classes are hit by $\aal$.
    \item $(24)-(29)$. All $\res$'s are $0$ since the top classes are divisible by $\aal$. $\tr$'s are again determined by whether the top classes are killed by $\aal$ or not. We use the multiplications involving $\tre$ (\ref{trelist}).
    \item $(31),(32)$. Determined by $\al$-linearity and $(5),(6)$.
\end{itemize}

Now for the Weyl actions. For simplicity, we use $\gamma$ for both of the actions on $G/C_2$ and $G/e$-level. We only need to worry about the $\Z$'s on the bottom two levels. From the discussion of (\ref{fixed}), any class that is hit by $\res$ is fixed by $\gamma$. Thus we are only left with $(15)$, $(22)$, $(24)$, $(27)$, $(33)$, $(36)$. However, we can notice that for all of these cases, the only thing we need to know to determine all the Weyl actions is that $\gamma\cdot \eal=-\eal$ and that $\gamma$ act as ring homomorphisms. Take $(15)$ as an example. From $(7)$ we know $\gamma\cdot 2\uts^{-1}=2\uts^{-1}$ since $2\uts^{-1}\cdot\widebar{\uta}$ is hit by $\res$. Thus the classes $2\uts^{-1}\cdot\eal^i$ has trivial action for even $i$ and $\gamma=-1$ for odd $i$.

The fact that $\gamma\cdot \eal=-\eal$ can be checked by a simple cellular computation: $\pi_{\alpha-1}^{G/C_2}H=H^1(S^{\alpha})(G/C_2)$. The latter is the cohomology of the following cochain after applying $(-)^{C_2}$:
\[
0\to\Z\xrightarrow{1+\widebar{\gamma}}\Z[G/C_2].
\]
Then $\gamma\cdot \ual=-\ual$ is deduced as a consequence.

\section{The \texorpdfstring{$\PP$}{pp}-homotopy limit spectral sequence for \texorpdfstring{$\pi_{\bigstar}F(E\PP_+,H)$}{pph}}
In this section we compute the $\PP$-homotopy limit spectral sequence for $\pi_{\bigstar}F(E\PP_+,H)$. Taking $\mathscr{F}=\mathscr{P}$ and $E=H\uZ$ in the generalized Tate square (\ref{FTate}), we get the following square:
\begin{equation}
	\xymatrix{
		E\PP_+\wedge H\ar[d]_{\simeq}\ar[r] &H\ar[d]\ar[r] &\HPf\ar[d]\\
		E\PP_+\wedge F(E\PP_+,H)\ar[r] &F(E\PP_+,H)\ar[r] &F(E\PP_+,H)[\aal^{-1}]
	}.
\end{equation}
As in subsection \ref{methodone}, the cellular filtration of $E\PP_+$ gives us the $\PP$-homotopy limit spectral sequence, of the form
\begin{equation}
	E_2^{V,s}=H^s(C_2';{\pi_V^{G/C_2}}H)\Rightarrow \pi_{V-s}^{G/G}(F(E\PP_+,H)),\,|d_r|=(r-1,r).\label{PforH}
\end{equation}
For consistency, we always use $RO(C_4)$-grading, and the coefficient ring of the group cohomology ring is
\begin{equation*}
    \pi_{*+*\alpha+*\lambda}^{G/C_2}H\cong \pi_{*+2*\sigma}^{C_2/C_2}H[\eal^{\pm}].
\end{equation*}
Note the entire spectral sequence including the target are $\uta$-local, thus we will use $\etal^{\pm}=\uta^{\mp}$ interchangeably. The action of $C_2'$ on the coefficient is determined by the parity of the power of $\eal$ with $\widebar{\gamma}\cdot\eal=-\eal$. So the coefficient ring splits as follows, where $\Z$ means trivial module, and $\widetilde{\Z}$ means that $\widebar{\gamma}$ acts as $-1$ on the generator.
\begin{Lem}
The coefficients of the group cohomology in (\ref{PforH}), as $\Z[C_2']$-modules, are
\begin{equation*}
    \begin{aligned}
    \pi_{*+*\alpha+*\lambda}^{G/C_2}H&=\Z/2[\ats,\uts]\langle\ats\rangle[\eal^{\pm}]\\
    &\oplus\Z[\uts][\etal^{\pm}]\\
    &\oplus\widetilde{\Z}[\uts][\etal^{\pm}]\langle\eal\rangle\\
    &\oplus\Z\langle 2\uts^{-i}\rangle_{i\geq 1}[\etal^{\pm}]\\
    &\oplus\widetilde{\Z}\langle 2\uts^{-i}\rangle_{i\geq 1}[\etal^{\pm}]\langle\eal\rangle\\
    &\oplus\Z/2\langle\Si\frac{1}{\ats^i\uts^j}\rangle_{i,j\geq 1}[\eal^{\pm}]
    \end{aligned}
\end{equation*}
\end{Lem}
We have the group cohomology
\begin{equation*}
   \begin{aligned}
   &H^*(C_2';\Z)=\Z[x]/2x,|x|=2.\\
   &H^*(C_2';\widetilde{\Z})=\Z/2\langle y\rangle[x],|y|=1.\\
   &H^*(C_2';\Z/2)=\Z/2[c],|c|=1.
   \end{aligned} 
\end{equation*}
Thus we deduce
\begin{Prop}
The $E_2$-term of the spectral sequence (\ref{PforH}) is
\begin{equation}
    \begin{aligned}
    E_2^{\bigstar,*}&=\Z/2\langle\ats^i\rangle_{i\geq 1}[\uts][\eal^{\pm}][c]\\
    &\oplus\Z[\uts][\etal^{\pm}][x]/2x\\
    &\oplus\Z/2[\uts][\etal^{\pm}]\langle\eal c\rangle[x]\\
    &\oplus\Z\langle 2\uts^{-i}\rangle_{i\geq 1}[\etal^{\pm}][x]/2x\\
    &\oplus\Z/2\langle 2\uts^{-i} \rangle_{i\geq 1}\langle\eal c \rangle[\etal^{\pm}][x]\\
    &\oplus\Z/2\langle \Si\frac{1}{\ats^i\uts^j} \rangle_{i,j\geq 1}[\eal^{\pm}][c].\label{E2ofH}
    \end{aligned}
\end{equation}
with
\begin{equation*}
    |c|=(0,1),|x|=(0,2),|\ats|=(-\lambda,0),|\uts|=(2-\lambda,0),|\eal|=(\alpha-1,0).
\end{equation*}
\end{Prop}
Here we have to note that the class $x\in E_2^{0,2}=H^2(C_2';\pi_0^{G/C_2}H)$ and the entire $E_2^{\bigstar,*}$ is a module over $E_2^{*,0}=\Z[x]/2x$. Actually, $x=c^2$ by looking at the $E_1$-page
\[
E^{\bigstar,*}_1=(\pi^{G/C_2}_{\bigstar}H)[c].
\]
We also have the relation
\begin{equation*}
    (\eal c)^2=\etal\cdot x.
\end{equation*}

The classes $\eal c,x$ are permanent cycles for degree reason, since $\pi^{G/C_2}_*H$ is concentrated at $*=0$ and $|d_r|=(r-1,r)$. Now let us fix the subring of $E_2$
\begin{equation*}
    A=\Z[\ats,\uts,\uta^{\pm},\eal c]/(2\ats,2\eal c).
\end{equation*}
We regard the entire spectral sequence as a spectral sequence of $A$-algebras. This is because $d_r$ is also $\uts,\ats,\uta^{\pm}$-linear as differentials in a spectral sequence of $F(E\PP_+,H)_{\bigstar}$-algebras. Note here that the classes $\ul,\al,\uta^{\pm}\in F(E\PP_+,H)_{\bigstar}(G/G)$ act on the $E_2$-page through $\res$.

Thus, for example, in the first summand of (\ref{E2ofH}), it's a module over $A$ with generators
\[
\ats,\ats\eal,\ats c.
\]

As in the computation of $\pi_{\bigstar}F(E\PP_+,\HPt)$ in section \ref{methodone}, using the cofiber sequences
\[
S(k\alpha)_+\to S^0\to S^{k\alpha}
\]
and their natuality with respect to $k\geq 1$, we can deduce all the differentials with the knowledge of $\pi_{\bigstar}H$.

\begin{Thm}
On the $A$-module generators, the permanent cycles in the spectral sequence (\ref{PforH}) are:
\begin{equation*}
    \begin{aligned}
    1, 2\uts^{-1},\Si\frac{1}{\ats^i\uts^j},\Si\frac{1}{\ats^i\uts}c,\Si\frac{1}{\ats^i\uts}\eal,\, \text{with}\,i,j\geq 1
    \end{aligned}
\end{equation*}
The non-trivial differentials are:
\begin{equation*}
    \begin{aligned}
    &d_3(\ats\eal)=\uts\eal cx,\, d_3(\ats c)=\uts x^2,\,d_3(\Si\frac{1}{\ats^i\uts^j}c)_{j\geq 2}=\Si\frac{1}{\ats^{i+1}\uts^{j-1}}x^2,\\
    &d_3(\Si\frac{1}{\ats^i\uts^j}\eal)_{j\geq 2}=\Si\frac{1}{\ats^{i+1}\uts^{j-1}}\eal cx,\, d_2(2\uts^{-j})_{j\geq 2}=\Si\frac{1}{\ats\uts^{j-1}}x.
    \end{aligned}
\end{equation*}
\end{Thm}

To illustrate the differentials, it is better to keep Figure \ref{E1} in mind. It shows the $E_1^{*+*\lambda,0}$. For the $\alpha$-grading, we make use of the $\eal$-periodicity (though be careful that $d_r$ is only $\etal^{\pm}$-linear). We have $d_r=(r-1,r)$. Thus $d_r$ preserve the power in the $\eal$'s. That is, 
\[
d_r(z\eal^k)=z'\eal^k \quad\text{with} z,z'\in E_r^{*+*\lambda,*}.
\]
This $\eal$-periodicity enable us to think of this spectral sequence as a $\Z$-graded tri-graded SS, indexed by the power of $\eal$. These SSs are independent. Since we have $\etal$-linearity, we actually only need to consider the case of $1$ and $\eal$, i.e, it is a $\Z/2$-graded tri-graded SS. $E_1^{*+*\lambda,*}$ consists of $\Z_{\geq 0}$ copies of the plane drawn in the picture above indexed by the cohomological dimension $s\geq 0$. $E_r^{*+*\lambda,*}$ is a subquotient of this structure. Similarly for $E_r^{*+*\lambda+\alpha,*}$.

\begin{center}
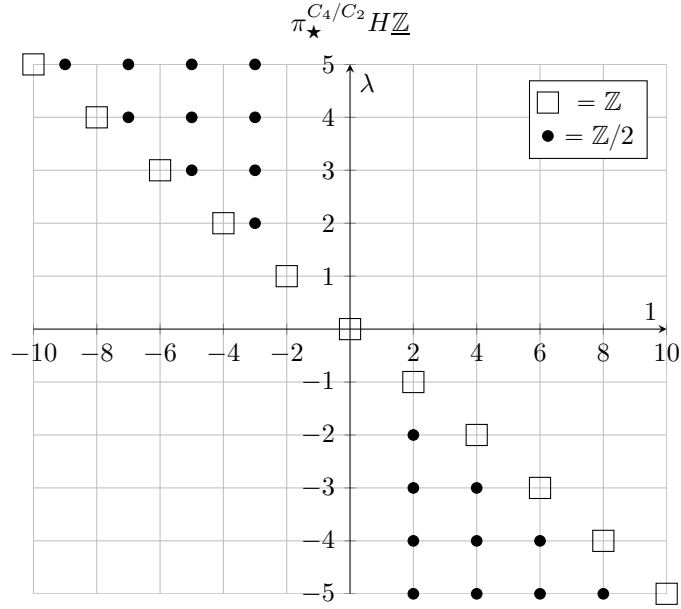
\begin{figure}
\begin{tikzpicture}
\begin{axis}[
    axis lines= center,
    title={$\pi^{C_4/C_2}_{\bigstar}H\uZ$},
    xlabel={1},
    ylabel={$\lambda$},
    xmin=-10, xmax=10,
    ymin=-5, ymax=5,
    xtick distance= 2,
    ytick distance= 1,
    legend pos= north east,
    ymajorgrids= false,
    grid = both,
]

\addplot[
    color=black,
    only marks,
    mark=square,
    mark size = 4pt
    ]
    coordinates {
    (-10,5)(-8,4)(-6,3)(-4,2)(-2,1)(0,0)(2,-1)(4,-2)(6,-3)(8,-4)(10,-5)
    };
    \legend{$=\Z$,$=\Z/2$}
    
\addplot[
    color=black,
    only marks,
    mark= *,
    mark size = 2pt
    ]
    coordinates {
    (-3,2)(-3,3)(-3,4)(-3,5)(-5,3)(-5,4)(-5,5)(-7,4)(-7,5)(-9,5)(2,-2)(2,-3)(2,-4)(2,-5)(4,-3)(4,-4)(4,-5)(6,-4)(6,-5)(8,-5)
    };
\end{axis}
\end{tikzpicture}
\caption{$E_1^{*+*\lambda,0}$}\label{E1}
\end{figure}
\end{center}

The picture right away tells us that the classes $1,\eal c,x,2\uts^{-1}$ survives to $E_{\infty}$ for degree reason. Other differentials are deduced using diagram chase. The spectral sequence collapse at $E_4$ and we have
\begin{Cor}
The $E_{\infty}$-page of (\ref{PforH}) is
\begin{equation*}
    \begin{aligned}
    E_{\infty}^{\bigstar,*}=E_4^{\bigstar,*}&=\Z/2[\ats,\uts,\etal^{\pm},\eal c]\langle\ats\rangle\\
    &\oplus\Z[\etal^{\pm},\eal c]/2\eal c\\
    &\oplus\Z[\uts,\etal^{\pm}]\langle\uts\rangle\langle 1,\eal c, x\rangle/(2\eal c,2x)\\
    &\oplus\Z\langle 2\uts^{-1}\rangle[\etal^{\pm},\eal c]/2\eal c\\
    &\oplus\Z\langle 4\uts^{-j}\rangle_{j\geq 2}[\etal^{\pm}]\\
    &\oplus\Z/2\langle\Si\frac{1}{\ats\uts^j}\rangle_{j\geq 1}[\etal^{\pm}]\langle 1,\eal c\rangle\\
    &\oplus\Z/2\langle\Si\frac{1}{\ats^i\uts^j}\rangle_{i\geq 2,j\geq 1}[\etal^{\pm}]\langle 1,\eal c,x\rangle\\
    &\oplus\Z/2\langle\Si\frac{1}{\ats^i\uts}c\rangle_{i\geq 1}[\etal,x]\\
    &\oplus\Z/2\langle\Si\frac{1}{\ats^i\uts}\eal\rangle_{i\geq 1}[\etal,x].
    \end{aligned}
\end{equation*}
\end{Cor}

\begin{Rem}
The classes $\Z\langle 4\uts^{-j}\rangle_{j\geq 2}[\etal^{\pm}]$ comes from the differentials
\begin{equation*}
    \begin{aligned}
        d_2(\langle 2\uts^{-j}\rangle_{j\geq 2}[\etal^{\pm}])=\Si\frac{1}{\ats\uts^{j-1}}\langle x\rangle[\etal^{\pm}].
    \end{aligned}
\end{equation*}
These are the map $\Z\xrightarrow{1}\Z/2$. The kernels are denoted by $\langle 4\uts^{-j}\rangle_{j\geq 2}[\etal^{\pm}]$. When multiplied by non-trivial powers of $\eal c$, the above differentials become the isomorphism $\Z/2\xrightarrow{1}\Z/2$.
\end{Rem}
For convergence, we use the notation $u\Rightarrow v,u\in E_{\infty},v\in \pi_{\bigstar}F(E\PP_+,H)$ to mean that the class $u$ converges to the class $v$.

\begin{Prop}
On the simple generators, we have:
\begin{equation*}
    \begin{aligned}
    &1\Rightarrow 1,\ats\Rightarrow\al,\uts\Rightarrow\ul,\etal^{\pm}\Rightarrow\uta^{\mp},\\
    &\eal c\Rightarrow\etal\aal,x\Rightarrow \etal \atal,2\uts^{-1}\Rightarrow 2\ul^{-1},4\uts^{-j}\Rightarrow 4\ul^{-j}
    \end{aligned}
\end{equation*}
For more complicated terms, we have
\begin{equation*}
    \begin{aligned}
    &\Z/2\langle 2\uts^{-1}\rangle[\etal^{\pm},\eal c]\langle\eal c\rangle\Rightarrow\Z/2\langle\frac{\aal^3}{\al}\rangle[\aal,\uta^{\pm}],\\
    &\Z/2\langle\Si\frac{1}{\ats^i\uts}c\rangle_{i\geq 1}[\etal,x]
    \oplus\Z/2\langle\Si\frac{1}{\ats^i\uts}\eal\rangle_{i\geq 1}[\etal,x]\Rightarrow\Z/2\langle\frac{\aal^3}{\al^2}\rangle[\al^{-1},\aal,\uta^{\pm}].
    \end{aligned}
\end{equation*}
\end{Prop}
We can see that the filtration is by power's of $\aal$, i.e., $F^s/F^{s+1}=\pi_{\bigstar}F(E\PP_+,H)\langle\aal^s\rangle$, which just corresponds to the fact that this is also the $\aal$-Bockstein spectral sequence.

The only non-trivial extensions that are involved are
\begin{equation*}
    \begin{aligned}
    &0\leftarrow\Z/2\langle\al\rangle[\uta^{\pm},\ul,\al]\xleftarrow{1}\Z/4\langle\al\rangle[\uta^{\pm},\ul,\al]\xleftarrow{2}\Z/2\langle\ul\aal^2\rangle[\uta^{\pm},\ul,\al]\leftarrow 0,\\
    &0\leftarrow\Z/2\langle\Si\frac{1}{\al^{i-1}\ul^{j+1}}\rangle[\uta^{\pm}]\xleftarrow{1}\Z/4\langle\Si\frac{1}{\al^{i-1}\ul^{j+1}}\rangle[\uta^{\pm}]\xleftarrow{2}\Z/2\langle\Si\frac{1}{\al^i\ul^j}\aal^2\rangle_{i\geq 2,j\geq 1}[\uta^{\pm}]\leftarrow 0.
    \end{aligned}
\end{equation*}
and we have
\begin{Thm} We have
\begin{equation}
    \begin{aligned}
    \pi_{\bigstar}^{G/G}F(E\PP_+,H)&=\Z[\uta^{\pm},\ul,\aal,\al]/(2\aal,4\al,\atal\ul-2\al\uta)\\
    &\oplus\Z\langle 2\ul^{-1}\rangle[\uta^{\pm}]\\
    &\oplus\Z\langle 4\ul^{-j}\rangle_{j\geq 2}[\uta^{\pm}]\\
    &\oplus\Z/2\langle\Si\frac{1}{\al^i\ul}\rangle_{i\geq 1}[\uta^{\pm}]\\
    &\oplus\Z/2\langle\Si\frac{\aal}{\al^i\ul^j}_{i,j\geq 1}\rangle[\uta^{\pm}]\\
    &\oplus\Z/4\langle\Si\frac{1}{\al^i\ul^j}\rangle_{i\geq 1,j\geq 2}[\uta^{\pm}]\\
    &\oplus\Z/2\langle\frac{\aal^i}{\al^j}\rangle_{i\geq 3,j\geq 1}[\uta^{\pm}]
    \end{aligned}
\end{equation}\label{PcompletionofH}
\end{Thm}
This coincides with $\pi_{\bigstar}^{G/G}H[\uta^{-1}]$. Actually, this computation is one of the motivations for us to think about the localization theorem proved in \cite{LocThm}.

\appendix
\section{Comparison with Nick Georgakopoulos' computation}
In this appendix we provide a comparison with Nick Georgakopoulos' computation by computer programs in \cite{NickG}. The two computations almost agree, except three typos in his final result as we will mention below. His computation has eight different parts, labeled as $6.1-6.8$. The Mackey functors inside each part are listed by bullets. We'll use notations like 
\[
6.4, L1 \,\text{with}\,\ul^{\leq -1}
\]
to refer to his classes listed in part $6.4$, Line $1$, that have power of $\ul$ less than or equal to $-1$. More precisely, his part $6.4, L1$ correspond to our classes $\Z/2\langle\Si\frac{\aal}{\al\uta}\rangle[\uta^{-1},\ul^{-1}]$, and the classes we are referring to are $\Z/2\langle\Si\frac{\aal}{\al\uta\ul^i}\rangle_{i\geq 1}[\uta^{-1}]$.

\begin{Rem}
Note that he uses $\sigma$ for the sign representation of $C_4$ while our notation for it is $\alpha$. We use $\sigma$ for the sign representation of $C_2\subset C_4$. Also we do not display the Weyl actions on the Mackey functors, since they are very easy to tell from the parity of the power of $\eal$.
\end{Rem}

Here is the comparison. We first translate some of his classes to ours. The correspondence is
\begin{center}
\begin{tabular}{ |c|c| } 
 \hline
  Ours &Nick Georgakopoulos'\\
 \hline
 $\Si\frac{1}{\ul\al}$ &$s$\\
 \hline
 $\tr(e_{n\alpha})$ &$w_n$\\
 \hline
 $\Si\frac{\aal}{\uta\al}=\tr(\ethal)\frac{\atal}{\al}$&$x_{1,1}$\\
 \hline
 $\Si\frac{\aal}{\uta\al}\uta^{-(n/2-1)}\ul^{-(m-1)}$&$x_{n,m}$\\
 \hline
\end{tabular}
\end{center}

Now we start the Mackey functors comparison. The left hand side is our Mackey functors numbered by $(1)-(38)$, the right hand side is his Mackey functors. The author will add comments in brackets.

{\scriptsize
\begin{center}
\begin{tabular}{|c|c|}
    \hline
    \text{Ours} &\text{Nick Georgakopoulos'}\\
    \hline
    $(1)\,\text{to}\,(4)+(33)+(34)$ from $\ual^{\geq 1}$ &$6.1+6.2$ (The positive cone)\\
    \hline
    $(5)$ &$6.3, L2$\\
    \hline
    $(6)$ &$6.5, L1$\\
    \hline
    $(7)$ &$6.7, L5$\\
    \hline
    $(8)$ &$6.3, L1$\\
    \hline
    $(9)$ &$6.7, L4$\\
    \hline
    $(10)+(16)$&$6.8, L3+ 6.4, L5$  (the range should be $3\leq i\leq m$ in 6.8 L3)\\
    \hline
    $(11)+(15)$&$6.8, L6+ 6.4, L1$\\
    \hline
    $(12)+(13)$&$6.3, L3+ 6.7, L2$\\
    \hline
    $(14)$&$6.7, L1$\\
    \hline
    $(17)$&$6.3, L4$ with $\aal^1$\\
    \hline
    $(18)$&Missing ($6.8, L4$ should be $0\leq i\leq \frac{n-3}{2}$ and then we get the corresponding classes)\\
    \hline
    $(19)$&$6.8, L1$\\
    \hline
    $(18)+(20)+(21)+(22)+(23)$&$6.8, L4+6.7, L3$ after the modification as in $(18)$\\
    \hline
    $(24)\,\text{to}\,(29)$&$6.3, L4\,\text{with}\,\aal^{\leq -1}+6.3, L5+6.4, L2+6.4, L6+6.4, L7+6.5, L3+6.6, L3+6.6,L5$\\
    \hline
    $(30)$&$6.6, L4$. (Note $2\aal^{-1}=\frac{\aal\ul}{\al\uta}$)\\
    \hline
    $(31)$&$6.5, L4$\\
    \hline
    $(32)$&$6.5, L2$\\
    \hline
    \hline
    $(33)$ with $i=0$ &$6.4, L3+6.2, L1\,\text{with}\, m=0$\\
    \hline
    $(27),(33) \,\text{with}\, i\geq 1$&$6.2, L1+6.6, L1$\\
    \hline
    $(34)$&$6.2, L2\,\text{with}\,\widebar{\ul}^{\geq 1}+ 6.6, L2$\\
    \hline
    $(35)$&$6.2, L2 \,\text{with}\,\widebar{\ul}^{0}+ 6.6, L6$ ($6.6, L6$ should be $n=1,m\geq 1$)\\
    \hline
    $(36)$&$6.8, L7$\\
    \hline
    $(37)$&$6.4, L4\,\text{with}\, \uts^{-1}+6.8,L2$\\
    \hline
    $(38)$&$6.4, L4\,\text{with}\, \uts^{\leq -2}+6.8, L5$\\
    \hline
\end{tabular}
\end{center}
}

\begin{Rem}
Note that some of our Mackey functors decompose into direct sums. This happens when $\tr=\res=0$. In this case, our Mackey functors decompose into direct sums of Mackey functors with trivial $G/G$-level and Mackey functors with trivial $G/C_2$ and $G/e$-levels. $(21),(22)$ are examples of this type. $(27),(28)$ are some extra examples. 

Our comparison above takes into account these decompositions. For example, in $(38)$, when we look at the elements
\[
\Z/2\langle \Si\frac{1}{\ats^i\uts^j}\rangle_{i\geq 1,j\geq 2}\cdot\eal^i,i\in \Z
\]
on the $G/C_2$-level and want to determine for what value of $i\in\Z$ do these elements not appear among $(1)-(32)$, we will notice that they appeared in $(12)$ and $(21)$, thus
the elements that have not been considered are
\[
\Z/2\langle \Si\frac{1}{\ats^i\uts^j}\rangle_{i\geq 1,j\geq 2}\langle u_{(2j-1)\alpha} \rangle[\etal].
\]
But a more careful look tells us that the Mackey functors in $(21)$ decomposes into direct sums of two summands. And taking this decomposition into account, the elements that are not 'covered' by $G/G$-level classes are
\[
\Z/2\langle \Si\frac{1}{\ats^i\uts^j}\rangle_{i\geq 1,j\geq 2}\langle \eal \rangle[\etal^{\pm}],
\]
i.e., we actually have Mackey functors of shape $(38)$ for all odd $i$.
\end{Rem}

\begin{Rem}
The author have used Nick Georgakopoulos' computer program for computing the additive structure to figure out the discrepancies.
\begin{enumerate}
    \item The author believes in his $6.8, L3$, the range should be modified to $3\leq i\leq m$. His $6.8,L3$ translates into $\Z/2\langle\Si\frac{\aal}{\al^s\ul^t}\rangle_{s\geq 1,t\geq 2}[\uta]$ in our notation. And his $6.8, L6$ translates into $\Z/2\langle\Si\frac{\aal}{\al\ul}\rangle[\ul^{-1},\uta]$ in our notation. The author believes in the first group of elements, the range should be $s\geq 2, t\geq 1$, which, in his notation, corresponds to $3\leq i\leq m$. That is because, firstly, the case $s=1,t\geq 2$ already appears in the second group, and his computer program indeed shows it is a single copy of $\langle \Z/2\rangle$ in each of these degrees. Secondly, $6.8, L3$ misses the classes in $s\geq 2,t=1$, and the computer program indeed shows we have a $\langle\Z/2\rangle$ in each of these degrees.
    \item The author believes in his $6.8, L4$ it should be $0\leq i\leq \frac{n-3}{2}$. His classes there are $\Z/2\langle\frac{\aal^i}{\al^i}\rangle_{i\geq 1,k\geq 7}[\uta]$. The author believes the classes $\Z/2\langle\frac{\aal^3}{\al}\rangle[\uta]$ and $\Z/2\langle\frac{\aal^5}{\al}\rangle[\uta]$ should exist. This is also verified by the computer program. For example, we have the Mackey functors $\underline{\pi_{\lambda-3\alpha+j(2-2\alpha)}}H\uZ=\langle\Z/2\rangle,j\geq 0$ by the computer program. 
    \item The suggested modification in $6.6, L6$ is also verified by the computer program.
\end{enumerate}
\end{Rem}

\section{Remarks on \texorpdfstring{$RO(G)$}{rog}-graded commutativity}\label{secgraded}
Here we provide a solid demonstration on why each level of our $C_4$-Mackey functors is a actual commutative ring, as opposed to a graded-commutative ring where there is a sign issue. Knowing the classes $\aal,\al,\uta,\ul$ commute without any sign enables us to carry out the entire computation without worry (even if this is not true, we can still do the computation by enforcing an order on the four classes). When we have the knowledge of the entire structure (\ref{answer}), the entire commutativity issue can be determined.

The $RO(G)$-graded commutativity of an equivariant commutative ring spectrum is systematically studied by Dan Dugger in \cite{Dugger14coherence}. See Proposition 6.13 and Remark 6.14 in that paper in particular. The application in our context is as follows: Take $a\in\pi_{i+j\alpha+k\lambda}H\uZ,b\in\pi_{m+n\alpha+o\lambda}H\uZ$, then we have
\begin{equation}
    a\cdot b=(-1)^{im}(-1)^{jn}b\cdot a\label{graded}
\end{equation}
This is because in Dugger's paper, the sign before $b\cdot a$ is determined by $({\tau}_{S^1})^{im}(\tau_{S^{\alpha}})^{jn}({\tau}_{S^{\lambda}})^{ko}$, where $\tau_X$ means the trace of identity of the $G$-spectrum $X$. Following \cite[Rem 8.6, p.171 and Rem 9.8, p.288]{LMS}, in $\pi_{\bigstar}S^0$, we have
\begin{equation*}
    \begin{aligned}
    &\tau_{S^1}=\chi(S^1)=-\chi(G/G)=-1\\
    &\tau_{S^{\alpha}}=\chi(S^{\alpha})=\chi(G/G)-\chi(G/C_2)=1-\res\circ\tr\\
    &\tau_{S^{\lambda}}=\chi(S^{\lambda})=\chi(G/G)-\chi(G/e)+\chi(G/e)=1.
    \end{aligned}
\end{equation*}
Here $1=id:S^0\to S^0$, $\res,\tr$ are the two maps between $G/G$ and $G/C_2$ which will become $\res,\tr$ on the Mackey functor level upon applying a contravariant functor like $[-,H\uZ]^G$. For the Euler characteristics, $\chi(G/G)$ is represented in the Burnside category $\mathscr{B}_G$ by the map $*\xleftarrow{}*\xrightarrow{}*$ as in Remark 9.8 in p.288 and it correspond to $id:G/G_+\to G/G_+$ in $Ho(Sp^G)$ by Prop 9.6 in p.286. $\chi(G/C_2)$ is represented by $*\xleftarrow{}G/C_2\xrightarrow{}*$ in $\mathscr{B}_G$ and correspond to $\res\circ\tr$ in $Ho(Sp^G)$. Under the Hurewicz map $\pi_{\bigstar}S^0\to \pi_{\bigstar}H\uZ$, we have
\begin{equation*}
    \tau_{S^1}=-1,\tau_{S^{\alpha}}=-1,\tau_{S^{\lambda}}=1.
\end{equation*}
Thus we get (\ref{graded}).

From this we deduce that in $\pi_{\bigstar}H\uZ$ every pair of elements commute without a sign: Firstly, $\aal,\al,\uta,\ul$ commute with every class. That is because $\aal$ is $2$-torsion while the other three elements have even grading of $1$ and $\alpha$. The same reasoning applies to most of the classes, and the only cases where we could possibly have sign issues are multiplications of $4$-torsion classes with a $\Si$ in front of them. But we have shown that these products are $0$. Commutativity of $G/C_2$-level and $G/e$-level are similarly determined.

\bibliographystyle{alpha}
\bibliography{bib}

\end{document}